\documentclass[11pt]{amsart}
\usepackage{graphicx}
\usepackage{amssymb}
\usepackage{epstopdf}
\usepackage{amsmath}
\usepackage{amscd}
\usepackage{amsthm,amsfonts,amsxtra,amssymb,amscd}
\usepackage[all]{xy}
\usepackage{enumerate}
\theoremstyle{plain}
\newtheorem*{lemma*}{Lemma}
\newtheorem{lemma}[subsection]{Lemma}
\newtheorem*{theorem*}{Theorem}
\newtheorem{theorem}[subsection]{Theorem}
\newtheorem*{proposition*}{Proposition}
\newtheorem{proposition}[subsection]{Proposition}
\newtheorem*{corollary*}{Corollary}
\newtheorem{corollary}[subsection]{Corollary}
\theoremstyle{definition}
\newtheorem*{definition*}{Definition}
\newtheorem{definition}[subsection]{Definition}
\newtheorem*{example*}{Example}
\newtheorem{example}[subsection]{Example}
\theoremstyle{remark}
\newtheorem*{remark*}{Remark}
\newtheorem{remark}[subsection]{Remark}

\newcommand{\R}{{\mathbb R}}

\newcommand{\CV}{{\mathcal{V}}}

\newcommand{\CE}{{\mathcal{E}}}

\newcommand{\g}{\mathfrak{g}}

\newcommand{\PW}{{\mathcal {P}}{\mathcal {W}}}
\newcommand{\C}{{\mathbb C}}
\newcommand{\Z}{{\mathbb Z}}
\renewcommand{\ll}{{\langle}}
\newcommand{\rr}{{\rangle}}

\newcommand{\ch}{{\mathrm{ch}}}

\newcommand{\Cone}{{\mathrm{Cone}}}
  
\newcommand{\Tr}{{\mathrm{Tr}}}

\newcommand{\Bott}{{\mathrm{Bott}}}

 \newcommand{\Thom}{{\mathrm{Thom}}}

\newcommand{\infdex}{{\mathrm{infdex}}}

 \renewcommand{\index}{{\mathrm{index}}}

\newcommand{\Ker}{{\mathrm{Ker}}}

\newcommand{\Coker}{{\mathrm{Coker}}}

 \newcommand{\ind}{{\mathrm{ind}}}

\newcommand{\CL}{{\mathcal{L}}}

 \renewcommand{\c}{{\mathfrak{c}}}

\title[Box spline]{
 Box splines and the equivariant  index theorem}
\author{ C. De Concini,\quad
C. Procesi,\quad M. Vergne}
\begin{document}

\newpage

\begin{abstract}
In this article,  we
 start to recall the inversion formula for the convolution with the Box spline.
 The equivariant cohomology  and the equivariant $K$-theory
with respect to a compact torus $G$ of  various spaces associated to a linear action of $G$  in a vector space $M$ can be both described  using some  vector spaces of distributions, on the dual of the group $G$ or on the dual of its Lie algebra $\g$.
 The morphism from $K$-theory to cohomology is analyzed and the multiplication  by the Todd class  is shown to correspond to the
 operator (deconvolution) inverting the  semi-discrete convolution with  a box spline.
 Finally, the multiplicities of the index of a $G$-transversally elliptic operator on $M$ are determined using the infinitesimal index of the symbol.
 \end{abstract}
 \maketitle

{\small \tableofcontents}

\section*{Introduction}

\subsection{Motivations}
The motivation of this work is to understand the multiplicities of
a representation of a torus $G$ in the virtual representation space $\Ker (A)-\Coker(A)$ obtained as the index of a $G$-invariant, elliptic or more generally transversally elliptic,  pseudo-differential operator $A$, in terms of the symbol. For basic definitions and results we refer to the Lecture Notes of   Atiyah (cf. \cite{At}).

We shall restrict to the case of a torus $G$ (with Lie algebra $\mathfrak g$) acting on a vector space (to some extent this is the essential case).
Let $\Lambda$ be the lattice of characters of $G$, and for $\lambda\in \Lambda$, we denote by $g\to g^{\lambda}$ the corresponding function on $G$.
 According to the theory of Atiyah-Singer (cf. \cite{At}),
in the case  of a transversally elliptic operator,
 $\Ker(A)$ and $\Coker(A)$ might be infinite dimensional, but the multiplicity of a character is finite and the difference of the two multiplicities  in $\Ker(A)$ and $\Coker(A)$ is the Fourier coefficient of the generalized function $\index(A)$ on $G$.
We thus obtain a function $\ind_m(A)$ on $\Lambda$  so that
${\rm
index}(A)(g)=\sum_{\lambda\in \Lambda} \ind_m(A)(\lambda) g^{\lambda}$. We call $\ind_m(A)$  the
multiplicity index map.

A cohomological formula for the equivariant index of elliptic operators was obtained by  Atiyah-Bott-Segal-Singer.
Using integrals of equivariant cohomology classes, a formula for the equivariant index of transversally elliptic operators was obtained in
 \cite{BV2}, \cite{BV1}, \cite{par-ver4}.
These formulae define (generalized) functions on $G$ in terms of the Chern character of the symbol of $A$.
However,  the behavior of multiplicities is our main interest.
Remark that, even in the case of elliptic operators where we deal with finite dimensional representations, a "formula" for the multiplicities is not easy to deduce from the Atiyah-Bott-Segal-Singer fixed point formulae, as  $\index(A)(g)$ is
given  by  different formulae for each $g\in G$.
 A similar drawback of the formulae of \cite{BV2}, \cite{BV1}, \cite{par-ver4} is that for each $g\in G$, they are
  defined only on a neighborhood of   $g$ (and with different formulae for each $g\in G$).
 Thus  known formulae for the equivariant index were  not adapted to the study  of the Fourier transform.

Our point of view is new.
Instead of   functions on $G$,  we    consider directly  the multiplicity index of an operator $A$  as a function on $\hat G$.
Similarly we associate directly to the Chern character of the symbol of $A$  a spline function on  $\mathfrak g^*$.
 Here splines (called also multisplines in several variables) are the familiar objects in approximation theory: piecewise polynomial functions with respect to a polyhedral subdivision  of $\g^*$
 (see \cite{deboor}).
 Our main theorem (Theorem \ref{final})  says (essentially)   that  the multiplicity index is the restriction of a suitable spline function  to $\Lambda$, a lattice in $\mathfrak g^*$.
Our inspiration comes from the "continuous analogue" of the index: the Duistermaat-Heckman measure, a piecewise polynomial function on $\mathfrak g^*$ and from the ``quantization commutes with reduction"  results   on multiplicities of twisted Dirac operators.
The key point of our approach are explicit computations of the index of some transversally elliptic operators in terms of vector partition functions. We construct  two piecewise polynomial  functions, one obtained from the multiplicity index and Box splines,
 the other from the Chern character using our theory of the infinitesimal index \cite{dpv3} and we compare them on generators.
 Finally,  our final theorem (Theorem \ref{final}) follows  from  a remarkable inversion formula, basically due to Dahmen-Micchelli \cite{DM1}, for multisplines.

\bigskip

%One typical instance of index formula  is the Riemann-Roch theorem for complex manifolds.  A particularly remarkable case is that
%  of toric varieties. For such a variety the Riemann-Roch theorem  for line bundles provides a dictionary between continuous and discrete formulae for polytopes: the number of integral points in a polytope is related to the volume by a Todd operator.
%\smallskip

Let us first recall the basic formalism of our approach.
 Let $M:=M_X=\oplus_{a\in X}L_a$ be a complex vector space with a linear action of $G$ where $a\in X\subset \Lambda$ is a character and $L_a$ denotes the corresponding 1--dimensional representation of $G$.

 The vector partition function $\mathcal P_X$, a function on $\Lambda$  which describes the multiplicity of the action of the torus $G$ on polynomial functions on $M$  is approximated by  a multispline distribution $T_X$: the convolution of the Heaviside functions associated to the half line $\R^+a$, where $a$ runs through the sequence $X$ of  weights of
 $G$ in $M$ (we assume here in the introduction that all weights $a$ are on one side of a half-space and span $\g^*$). The  locally polynomial measure $T_X$  on $\g^*$ is  the Duistermaat-Heckman measure of the Hamiltonian vector space $M_X$.

 In approximation theory, one  introduces another special distribution {\em  the Box spline $B_X$} defined as  convolution  of the intervals $[0,1]a$ (thought of as measures or distributions).
An immediate  relation between $\mathcal P_X$ and $B_X$  is the fact that the  convolution of the Box spline $B_X$ with the partition function  $\mathcal P_X$ is the multispline $T_X$.  The Todd operator, an infinite series of constant coefficients differential operators,
 acts on spline functions. It  enters naturally in the  "deconvolution" formula,
  leading to the "Riemann-Roch formula" for $\mathcal P_X$ in function of $T_X$ (at least in the special case of $X$ unimodular):
We  apply  a series of constant coefficient operators to the piecewise polynomial function $T_X$ and then restrict it to the lattice. In this way we obtain the  vector partition function $\mathcal P_X$.

  These algebraic formulae are well-known:
cf.  Khovanskii-Pukhlikov \cite{KP}, Dahmen-Micchelli \cite{DM3}, Brion-Vergne \cite{brionvergne}, De Concini--Procesi \cite{dp1}
and they are equivalent to the Riemann-Roch theorem for line bundles over toric varieties.

  Our aim in this article is to show that the same deconvolution formula allows us to compute the index of any transversally elliptic operator  on $M_X$ in function of a piecewise polynomial  function on $\g^*$ associated to its symbol by applying to it the Todd  differential operator.

\subsection{Summary of results}   Let $M$ be a manifold, $T^*M$ its cotangent bundle and $p:T^*M\to M$ the canonical projection.  Given now a pseudo differential operator $A$ between the sections of two vector bundles $\mathcal E^+,\mathcal E^-$,  one constructs its symbol $\Sigma=\Sigma(x,\xi)$  which is a bundle map $\Sigma:p^*\mathcal E^+\to p^*\mathcal E^-$.

If $M$ has a
  $G$ action,  we
 denote by $T^*_GM$ the closed subset of $T^*M$, union of the conormals to the $G$ orbits.  Then a $G$--equivariant pseudodifferential operator $A$  is called
  $G$-transversally elliptic if the symbol $\Sigma$ restricted to  $T^*_GM$ minus the zero section is an isomorphism of bundles.\smallskip

 The symbol $\Sigma(x,\xi)$ of the pseudo-differential  transversally elliptic operator $A$  on $M$ determines two topological objects:

 1) An element of the equivariant $K-$theory group  $K_G^0(T^*_GM).$

2)  The Chern character
$\ch(\Sigma)$ of $\Sigma$, which   is  an element of the $G$-equivariant cohomology with compact supports of $T^*_GM$.
\medskip

 The index    of $A$, denoted $\index(A)$,  depends only on the symbol and defines a map from $K_G^0(T^*_GM) $ to the space of  generalized functions on $G$. The Fourier transform of $\index(A)$ is the multiplicity index map $\ind_m(A)$, a function on $\Lambda\subset \g^*$.

In \cite{dpv3}, we have associated  to $\ch(\Sigma)$   a distribution on $\mathfrak g^*$, its infinitesimal index, denoted
$\infdex(\ch(\Sigma))$.

Assume now that $M$ is a real vector space with a linear action of $G$.
 The list of weights of $G$ in the complex vector space $M\otimes_{\R}\C$ is $X\cup -X$ for some list $X\subset \Lambda$.
 For simplicity assume that $X$ generates $\mathfrak g^*$.

  For any $\Sigma\in K_G^0(T^*_GM)$, we prove that
the  distribution $\infdex(\ch(\Sigma))$
 is piecewise polynomial on $\mathfrak g^*$.
 Furthermore (Theorem \ref{theo:boxindex}), the following identity of locally $L^1$-functions of $\xi\in \mathfrak g^*$ holds

 \begin{equation}\label{pp}
\sum_{\lambda\in \Lambda}\ind_m(A)(\lambda) B_{X\cup -X}(\xi-\lambda)=\frac{1}{(2i\pi)^{\dim  M}}\infdex(\ch(\Sigma))(\xi).
\end{equation}
 In other words, $\infdex(\ch(\Sigma))$ is (up to a multiplicative constant)  the spline function obtained from convoluting  a Box spline  with the discrete measure $\sum_{\lambda}\ind_m(A)(\lambda)\delta_\lambda$.
 It is easy to check  this formula on  the Atiyah symbols $At^F$, and these elements  generate the $R(G)$-module $K_G(T^*_GM)$.
 In some sense, as explained  in Part 2, this formula is the Fourier transform  of   the formulae of Berline-Paradan-Vergne.

Apply the Todd operator associated to $X\cup -X$ to the spline function  $\infdex(\ch(\Sigma))(\xi)$. We obtain again a spline function $p$ on $\mathfrak g^*$.
Our final result (Theorem \ref{final})
says essentially that the restriction of this function $p$ to $\Lambda$ is the multiplicity index.
This follows from a   "deconvolution" formula for splines functions produced by convolution with Box splines.

\subsection{Outline of the article}

$\bullet$
In Part 1,  we   recall  some
  results obtained by Dahmen-Micchelli in the purely combinatorial context of the semi-discrete convolution with the Box spline.
Let  $X=[a_1,a_2,\ldots,a_N]$ be a finite list of vectors in a vector space $V$. The Box spline $B_X$ is the image measure of the hypercube $[0,1]^N$ by the map $(t_1,t_2,\ldots, t_N)\to \sum_i t_i a_i$.

We assume next that $X$ span a lattice $\Lambda$.
Convoluting a discrete measure supported on the lattice $\Lambda$ by $B_X$ produces  a locally polynomial function on $V$.
The first step is to   prove a {\em deconvolution  formula}.  Define  the Todd operator
$\prod_{a\in X}\frac{\partial_a}{1-e^{-\partial_a}}$.
Then we prove in Theorem \ref{theo:inversionuni} that, if we apply (in an appropriate sense) the Todd operator $Todd(X)$ to the Box spline and restrict to the lattice, we obtain the $\delta$ function of the lattice $\Lambda$ in the case of a unimodular system
(a  slightly more complicated  formula is obtained for any $X$):

 $$Todd(X)*B_X|_\Lambda=\delta_0.$$

We prove this result using our knowledge of the Dahmen-Micchelli spaces (see \cite{dp1},\cite{dpv1}).

We show that
Khovanskii-Pukhlikov formula and
more generally  Brion-Vergne formula for the partition function $\mathcal P_X$  is a particular case of
the  deconvolution formula.

$\bullet$
In Part 2, we consider $M:=M_X$ as a real $G$-manifold, we recall   our description of $K_G^0(T^*_GM)$ as well as $H_{G,c}^*(T^*_GM)$ as vector spaces of distributions, on the dual of the group $G$ or on the dual of its Lie algebra $\g$ (cf. \cite{dpv2} and \cite{dpv4} respectively). We compute the infinitesimal index of the  Chern character of the Atiyah symbol $At^F$ and descent formulae associated to
a finite subset of  elements of $g\in G$.
We use all these ingredients to give a general Formula, in  \eqref{lamax}, for the index of a transversally elliptic operator $A$ in function of the infinitesimal index  of the Chern character of the symbol of $A$.

\bigskip

 As we already pointed out,  our results are motivated by previous results of  Berline-Vergne, and Paradan-Vergne.
 But our point of view is dual. We work with functions on $\hat G$ or $\g^*$.
  This way, we  are dealing with very familiar objects: the  partition functions and the multispline functions.
 It is remarkable that the transversally elliptic operators
 having multiplicity index Partition functions are the building blocks of "all index theory".

We thank Michel Duflo and Paul-Emile Paradan for various comments on this text.

\section{Notations and preliminaries\label{Nop}}
Let $V$ be a $s$-dimensional real vector space equipped with a lattice $\Lambda\subset V$.   We choose   the Lebesgue measure $dv$ on $V$ for which   $V/\Lambda$ has volume 1. With the help of $dv$, we  can freely identify generalized functions on $V$ and distributions on $V$. We denote by $\C[V]$ the space of (complex valued) polynomial functions on $V$.

We denote by $S^1$ the circle group of complex numbers of modulus 1. The character group   of $\Lambda$, $G:=\hom(\Lambda,S^1)$  is   a compact torus, and $V$ can be identified to $\mathfrak g^*$, the dual of the Lie algebra $\mathfrak g$ of $G$.   Of course if we choose a basis of the lattice $\Lambda$, then we may identify  $\Lambda$ with $\Z^s$,  $V$ with $\R^s$ and $G$ with $(S^1)^s$.

Dually,   let    $\Gamma\subset \mathfrak g=V^*$ be the  lattice of elements $x\in \g$ such that $ \langle x\,|\, \lambda\rangle\in 2\pi\mathbb Z $ for all $\lambda\in \Lambda$, the torus $G$ is $\g/\Gamma$.
If $x\in \mathfrak g$, we shall denote by $e^x$ its class in $G$ and the duality pairing $G\times \Lambda\to S^1$ will be given by
$$(e^x,\lambda)\mapsto e^{i \langle x\,|\, \lambda\rangle}.$$
Under this duality, $\Lambda$ is the character group of $G$ and will sometimes be denoted by $\check G$.

We identify the space $C^\infty(G)$ with the subspace of $C^\infty(\mathfrak g)$ formed by functions periodic under $\Gamma$.  If $g=e^x\in G$, we will sometimes write   $ g^{\lambda}$ for   $e^{i \langle x\,|\, \lambda\rangle}$.

  $L_\lambda$ will denote the one-dimensional complex vector space with action of $G$ given by $g^{\lambda}.$
 Notice that, as a real $G$ linear representation, $L_\lambda$ is isomorphic to $L_{-\lambda}$ by changing the complex structure with the conjugate one.

More generally,
 \begin{definition}\label{MX}
Let $X$ be a finite sequence of non zero elements of  $\Lambda$.
Define  the  vector space
\begin{equation}\label{SMX}
M_X:=\oplus_{a\in  X}L_a.
\end{equation}
\end{definition}

Thus $M_X$ is a complex representation space for $G$ and every finite dimensional complex representation of $G$ is of this form for a well defined $X$.

Again the  space $M_X$ as a real $G$- representation depends only of the sequence $X$ up to sign changes.
In fact it will be very important to consider for a real representation space $M$ (with no $G$-invariant non zero vector)  all possible $G$-invariant complex structures on $M$.

 The space of $\mathbb C$-valued functions on $\Lambda=\check G$ will be denoted by $\mathcal C[ \Lambda ]$,  while we shall set $\mathcal C_\Z[\Lambda ]$ to be the subgroup of $\Z$-valued functions. We display such a function $f(\lambda)$ also as a formal series
$$\Theta(f):=\sum_{\lambda\in\Lambda}f(\lambda) e^{ i\lambda }.$$

The subspace $\mathbb C[\Lambda]$ of the functions with finite support is  the group algebra of $\Lambda$ but can be also considered as the coordinate ring of the complex torus $\mathfrak g\otimes_{\mathbb R} \mathbb C/\Gamma$ as algebraic group. Finally $\mathbb Z[\Lambda]:=\mathbb C[\Lambda]\cap \mathcal C_\Z[\Lambda]$ is the group ring of $\Lambda$ but can also be considered as the character ring of $G$ or the Grothendieck group of finite dimensional representations of $G$. Due to this, we shall sometimes denote it by $R(G)$. Indeed, if $T$ is a representation of $G$ in a  finite dimensional complex vector space, then $\Tr\, T(g)$ is a finite linear combination of characters and this gives the desired homomorphism.

If $f(\lambda)$ is of at most polynomial  growth, the  series
$g\to \sum_{\lambda\in\Lambda}f(\lambda) g^{\lambda }$
defines  a  generalized function on
the torus $G$. We  denote by  $R^{-\infty}(G)$ the  subspace of $\mathcal C[\Lambda]$ consisting of these $f(\lambda)$.

Let us point out that $\Lambda$ acts on $\mathcal C[ \Lambda ]$ by translations, namely if $a\in \Lambda$ and $f\in \mathcal C[ \Lambda ]$,
$(t_af)(\lambda):=f(\lambda-a)$. This clearly corresponds to multiplication by $e^{ia}$ on $\Theta(f)$. It follows that both
$\mathcal C[ \Lambda ]$ and  $R^{-\infty}(G)$ are $\mathbb C[ \Lambda ]$-modules and of course $ \mathcal C_\Z[\Lambda]$ is a $\mathbb Z[\Lambda]$-module. We also define the difference operator $$\nabla_a:=id-t_a.$$

Passing to the continuous setting, if we take the space of polynomial functions $\mathbb C[\mathfrak g]$ on $\g$  (equal to the symmetric algebra $S[\mathfrak g^*]$),  we are going to consider  the space of distributions $ \mathcal D'(\mathfrak g^*)$ on $\mathfrak g^*$ as a $S[\mathfrak g^*]$-module, using differentiation. We denote by $\partial_a$ the partial derivative in the $a\in\mathfrak g^*$ direction.

\part{Algebra}

\section{Box splines}

\subsection{Splines}
    Let $X=[a_1,a_2,\ldots, a_N]$ be   a sequence (a multiset)  of  $N$ non zero vectors in $\Lambda$.

    The {\bf zonotope} $Z(X)$
associated to $X$ is the polytope
  $$
  Z(X):=\{\sum_{i=1}^N  t_i a_i\,\,|\,\, t_i\in [0,1]\}.
  $$
In other words, $Z(X)$ is the Minkowski sum of the segments $[0,a_i]$
over all  vectors $a_i\in X$.

%We denote by $\Cone(X)$
%the cone generated by the elements $X$ in $V$.
%The  cone $\Cone(X)$ is equal to  $V$  if and only if  $0$ is in the interior of the zonotope $Z(X)$.
%
%

%We denote by $\C[V]$ the space of (complex valued) polynomial functions on $V$.

    Recall that the Box spline $B_X$ is the distribution on $V$  such that, for a test function
    $test$ on $V$, we have the equality

    \begin{equation}\label{eq:box}
 \ll B_X, test\rr =\int_{t_1=0}^1\cdots \int_{t_N=0}^1 test(\sum_{i=1}^N t_i a_i) dt_1\cdots dt_N.
\end{equation}

The Box spline is a compactly supported probability measure on $V$ and we have

\begin{equation}\label{eq:FB}
\int_V e^{i\ll v,x\rr}B_X(v)=\prod_{k=1}^N \frac{e^{i\ll a_k,x\rr }-1}{i\ll a_k,x\rr }.
\end{equation}

If $X$ generates $V$, the zonotope is a full dimensional polytope, and $B_X$  is given by integration against a piecewise polynomial function on $V$, supported and continuous on $Z(X)$, that we  still call $B_X$.

\begin{example}Let $V=\R$ be  one dimensional  and let $X_k=[1,1,\ldots,1]$ where $1$ is repeated $k$ times.
Figure \ref{b1} gives  the graphs of $B_{X_1}$, $B_{X_2}$ and  $B_{X_3}$.

\begin{figure}
\begin{center}
 \includegraphics[width=37mm]{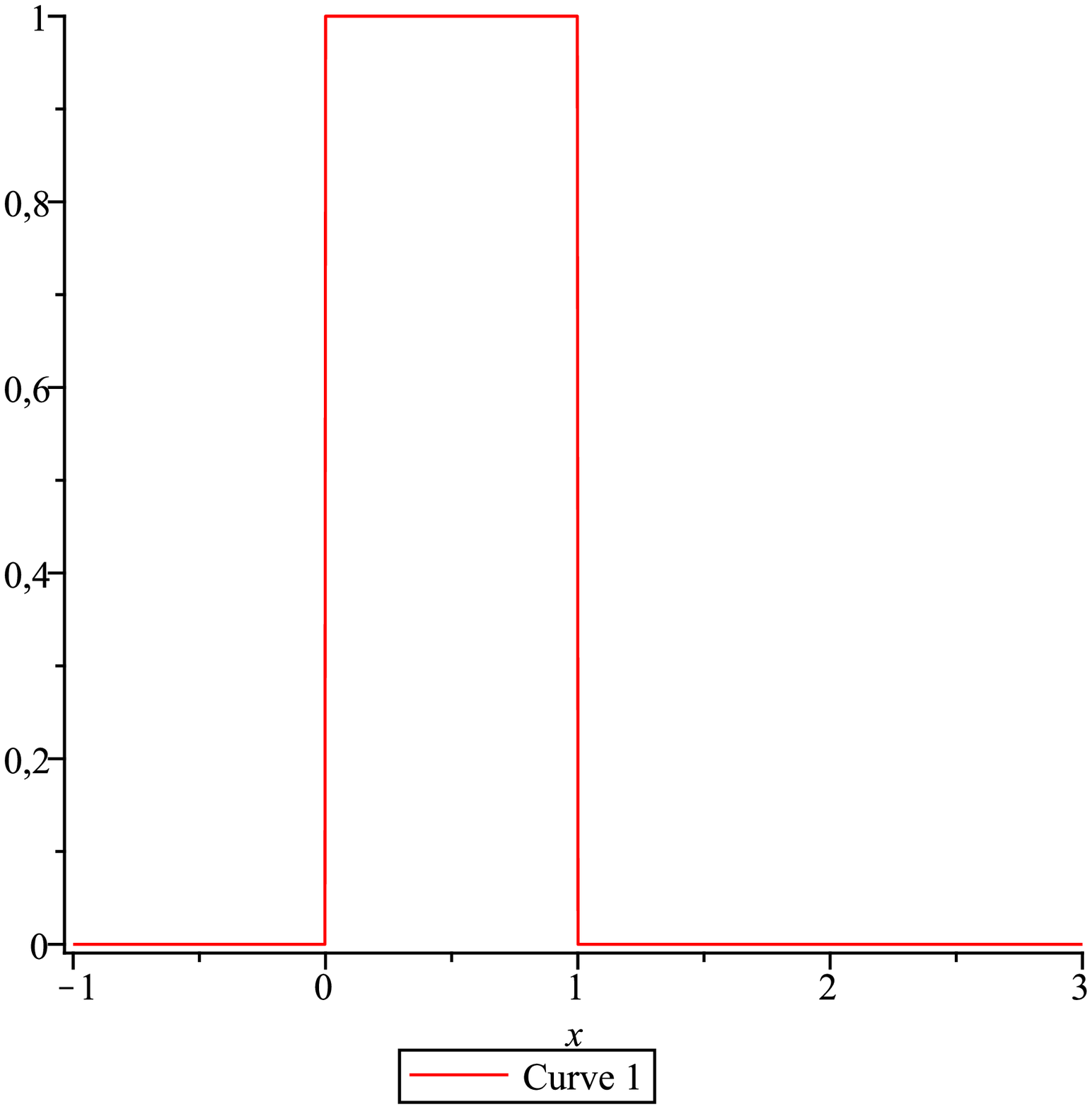} \includegraphics[width=37mm]{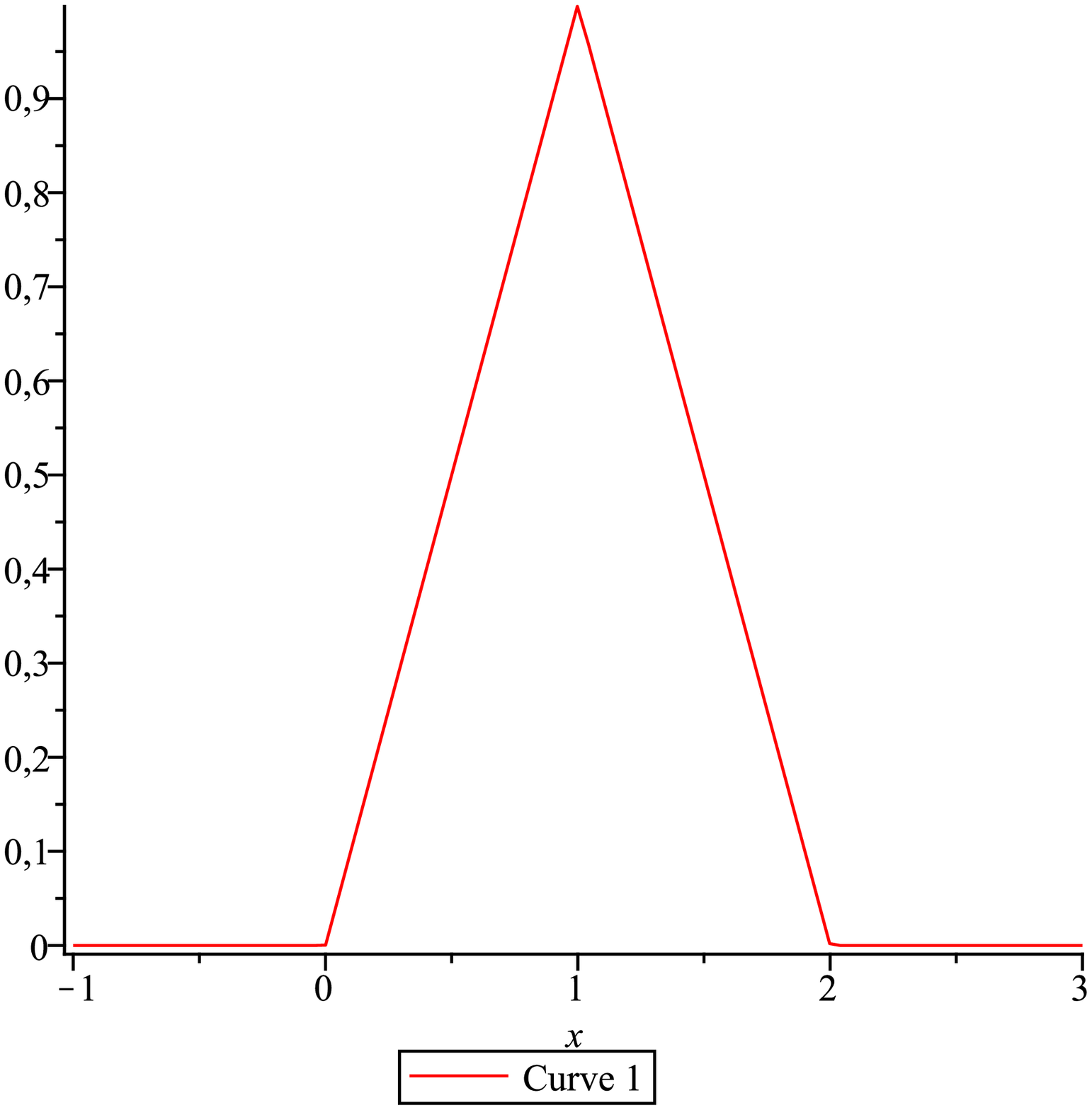}  \includegraphics[width=37mm]{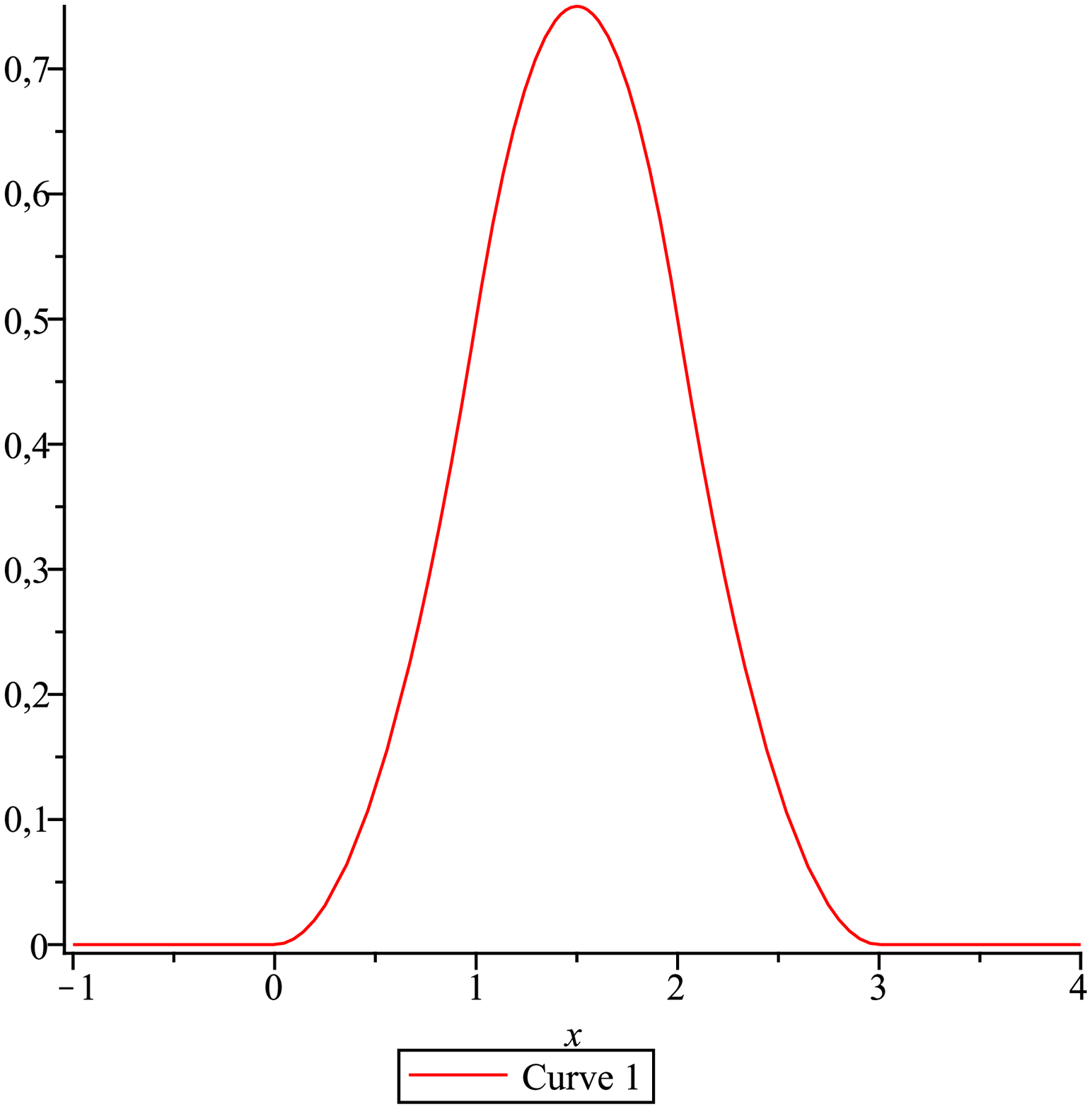}\\
 \caption{$B_{X_1},\quad B_{X_2},\quad B_{X_3}$}\label{b1}
 \end{center}
\end{figure}

\end{example}

Let us describe more precisely where this function is given by a polynomial formula.
\begin{definition}
An hyperplane of $V$ generated by a subsequence of elements of $X$ is
called  {\em admissible}.

An {\em affine admissible hyperplane} is a translate  $\lambda+H$ of an admissible hyperplane $H$ by an element $\lambda\in \Lambda$.
\end{definition}

\begin{remark}
The zonotope is bounded by affine admissible hyperplanes.
\end{remark}

\begin{definition}
An element of $v\in V$ is called  {\em regular} if  it   does not lie in any admissible hyperplane. We denote by $V_{\rm reg}$ the open subset of  $V$ consisting of  regular elements.
A connected component $\c$ of the set of regular elements will be called a {\em (conic) tope}.

An element of $v\in V$ is called  {\em affine regular} if it   does not lie in any  admissible affine hyperplane. We denote by $V_{\rm reg,aff}$ the open subset of  $V$ consisting of affine regular elements.
A connected component $\tau$ of the set of affine regular elements will be called an  {\em alcove} (see Figure \ref{a2}).

\end{definition}

\begin{figure}
\begin{center}
 \includegraphics[width=37mm]{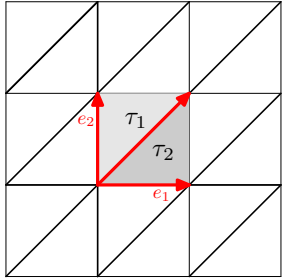}\\
 \caption{Topes for $X:=[e_1,e_2,e_1+e_2]$}\label{a2}
 \end{center}
\end{figure}

\begin{definition}

We will say that a  locally $L^1$ function $b$ on $V$ is piecewise polynomial (with respect to $(X,\Lambda)$)  if, on each alcove $\tau$, there exists a polynomial function $b^{\tau}$ on $V$ such that the restriction of $b$ to $\tau$ coincides with the restriction of  the polynomial $b^{\tau}$ to $\tau$.

If $b$ is a piecewise polynomial function,  we will say that the distribution $b(v)dv$ is piecewise  polynomial.

We denote by $\PW_{(X,\Lambda)}(V)$ the space of these piecewise polynomial functions on $V$.

When there is no ambiguity, we may drop $\Lambda$ or $X$ or both and write simply  $\PW_{ X }(V)$ or $\PW (V)$.
\end{definition}

\begin{remark}\begin{itemize}
\item $\PW_{(X,\Lambda)}(V)$ is preserved by the translation operators, $t_a$, $a\in \Lambda$.
\item The polynomial function $b^{\tau}$ is uniquely determined by $b$ and $\tau$.

\item The support of a piecewise polynomial function is a union of closures of alcoves.

\item      If $X$ generates $V$, the Box spline  $B_X$  is a piecewise  polynomial  function supported on the zonotope $Z(X)$.
Furthermore,  if  for any $a$ in $X$, $X\setminus\{a\}$ still spans $V$, this piecewise polynomial function extends continuously on $V$. In particular this applies if $0$ is an interior point in $Z(X)$.
\end{itemize}

\end{remark}

\bigskip

A piecewise polynomial function $h$ maybe continuous. In this case, its restriction to the lattice $\Lambda$ is well defined. If not, we may define  the "restriction" of $h$ to $\Lambda$ by a limit procedure as follows.

 Consider a piecewise polynomial function $h$ on $V$.
Let $\c$ be an alcove in $V$ containing $0$ in its closure.
Then for any $\lambda\in \Lambda$, $\tau=\lambda+ \c$ is an alcove on which $h$ is the polynomial $h^{\tau}$. We thus  define  a  map  $\lim_{\c}: \PW(V)\to \mathcal C[\Lambda] $ by setting
$$\lim_{\c}(h)(\lambda):=h^\tau(\lambda).$$

\begin{example}
Consider the box spline $B_{X_1}$ in Figure \ref{b1}. If $\c$ is the alcove $]0,1[$, then  $\lim_{\c}B_{X_1}(0)=1$, as we take the limit from the right, while  $\lim_{-\c}B_{X_1}(0)=0$, for the opposite alcove.
\end{example}

Notice that the     operator $\lim_{\c}$ on $\PW(V)$ commutes with translations by elements of $\Lambda$.

It is convenient to think of an element in $\PW_{(X,\Lambda)}(V)$ as a function only on the set of affine regular points. As such, differentiating alcove by alcove,  this space is a module over the ring of formal differential operators of infinite order with constant coefficients.

Therefore we may set \begin{definition}\label{oppw}
Given an operator $D$ of infinite order with constant coefficients  and $b\in \PW_{(X,\Lambda)}(V)$ we shall denote by   $D_{pw}b\in \PW_{(X,\Lambda)}(V)$  the element defined by the action of $D$ on $b$ alcove by alcove:
$$(D_{pw}b)^{\tau}=Db^{\tau}$$
\end{definition}

Notice that the action $D_{pw}$ on $\PW(V)$ commutes with the action of translation by elements of $\Lambda$.
 type.

{\bf Warning}\quad We may act on such a  function,  thought of as distribution,  with a finite order differential operator. In general we get a different result to that we obtain taking this function of affine regular points, applying the same operator and then considering the result as a $L^1$ function. Indeed the two coincide only on the set of affine regular points.

\begin{example}
Consider again the Box spline $B_{X_1}$ from Figure \ref{b1}. Then if we consider $B_{X_1}$ as a distribution, we have $\partial B_{X_1}=\delta_0-\delta_1$, a difference of two delta functions, while $\partial_{pw} B_{X_1}=0$.
\end{example}

\medskip

If $K$ is a polynomial function on $V$, we will say that the function $\lambda\mapsto K(\lambda)$ is a polynomial function on $\Lambda$.  The polynomial function $K$ is determined by its restriction to $\Lambda$.
 A function $k$ on $\Lambda$ for which there exists a sublattice $\Lambda'$ of $\Lambda$ such that, for any $\xi\in \Lambda$, the function
$\nu \to k(\xi+\nu)$  is polynomial on $\Lambda'$ will be called a quasi polynomial function on $\Lambda$.

We denote by $\delta_0$ the function on $\Lambda $ identically equal
to $0$ on $\Lambda $,  except  for $\delta_0(0)=1$.

\begin{definition}

If $f\in \mathcal C[\Lambda]$, define
 the distribution $B_X*_d f$ by
 $$B_X*_d f=\sum_{\lambda\in \Lambda} f(\lambda) t_{\lambda}B_X.$$

When $X$ spans, this gives rise  to the piecewise  polynomial function
 $$(B_X*_d f)(v)=\sum_{\lambda\in \Lambda} f(\lambda)  B_X(v-\lambda).$$
\end{definition}

The notations $*_d$ means discrete.
$B_X*_d f$ is the convolution of $B_X$ with the discrete measure $\sum_\lambda f(\lambda)\delta_\lambda$.

We denote (to emphasize the difference with the discrete case) the usual convolution
of two distributions $\theta_1,\theta_2$
(with some support conditions so that their convolution exists)
 by $\theta_1*_c\theta_2$.

Our aim is to write an inversion formula for $f\to B_X*_d f$.
As this operator is not injective, we will need a few other data.

\begin{remark}If $p\in \C[V]$ is a polynomial, then by Taylor formula, we have $t_b p=e^{-\partial_b}p$.\end{remark}

If $Y$ is a  sequence of vectors, we define
the operator $I(Y)$  on $\C[V]$ by
\begin{equation}
I(Y):=\prod_{a\in Y} \frac{(1-e^{-\partial_a})}{\partial_a}.
\end{equation}
Then, by integrating the Taylor formula,  we have $$\int_{t_1=0}^1\cdots \int_{t_N=0}^1 p(v-(\sum_{i=1}^N t_i a_i)) dt_1\cdots dt_N= (I(Y)p)(v).$$

The operator  $I(Y)$ is an invertible operator on $\C[V]$.
 We denote the inverse of $I(Y)$  by $Todd(Y)$:
\begin{equation}
Todd(Y):=\prod_{a\in Y} \frac{\partial_a}{(1-e^{-\partial_a})}.
\end{equation}

Notice that  $$\ll B_X*_cf,test\rr =
   \int_{t_1=0}^1\cdots \int_{t_N=0}^1 \langle f, test(v-(\sum_{i=1}^N t_i a_i)) \rangle dt_1\cdots dt_N$$
  for any distribution $f$ on $V$.
Thus \begin{proposition}
\label{dire}For $p$ a polynomial function on $V$, the usual convolution $B_X*_c p$   is still a polynomial given by
the formula
$B_X*_cp=I(X)p$
and its inverse is given by the operator $Todd(X)$.
\end{proposition}

\subsection{Inversion formula: the unimodular case}

Recall that a sequence $X$ is unimodular if $X$ spans $V$ and if any basis $\sigma$ of $V$ extracted from $X$ is a basis of $\Lambda$.
We will prove now that if $X$ is unimodular, then the inverse of the semi-discrete convolution by the box spline $B_X$  $$K\to B_X*_d K; \hspace{1cm}\C[\Lambda]\to \PW(V)$$
is  obtained by applying the operator $Todd(X)_{pw}$ (cf. Definition \ref{oppw}) to the piecewise polynomial function $B_X*_d K$ and then passing to a suitable limit.

\begin{theorem}\label{theo:inversionuni}
 Assume that $X$ is unimodular. Let $\c$ be an alcove in $V$ containing $0$ in its closure and contained in $Z(X)$.
  Then
\begin{enumerate}[i)]
\item
$\lim_{\c} (Todd(X)_{pw}B_X)=\delta_0.$

\item
For any $K\in \mathcal C[\Lambda]$,

$$K= \lim_{\c} (Todd(X)_{pw} (B_X*_d K)).$$
\end{enumerate}
\end{theorem}

\begin{remark}

If $0$ belongs to  the interior of the zonotope $Z(X)$, one can show that  the function  $Todd(X)_{pw} B_X$ extends to a continuous function on $V$. We will recall this result in Remark \ref{cont}.
\end{remark}

\begin{remark}\label{soeq}
The two items  in Theorem \ref{theo:inversionuni} are equivalent statements. The first item is the particular case of the second item applied to $K=\delta_0$, and the other is deduced from the first one by writing $K$ as a linear combination   of translates of $\delta_0$.
However, we list them  independently  as  we want to emphasize  this striking property of the Box spline function.
 Figure   \ref{tb1}
 describes $(Todd(X)_{pw}B_X)$ for $X=X_1,X_2,X_5$.

 \begin{figure}
\begin{center}
 \includegraphics[width=37mm]{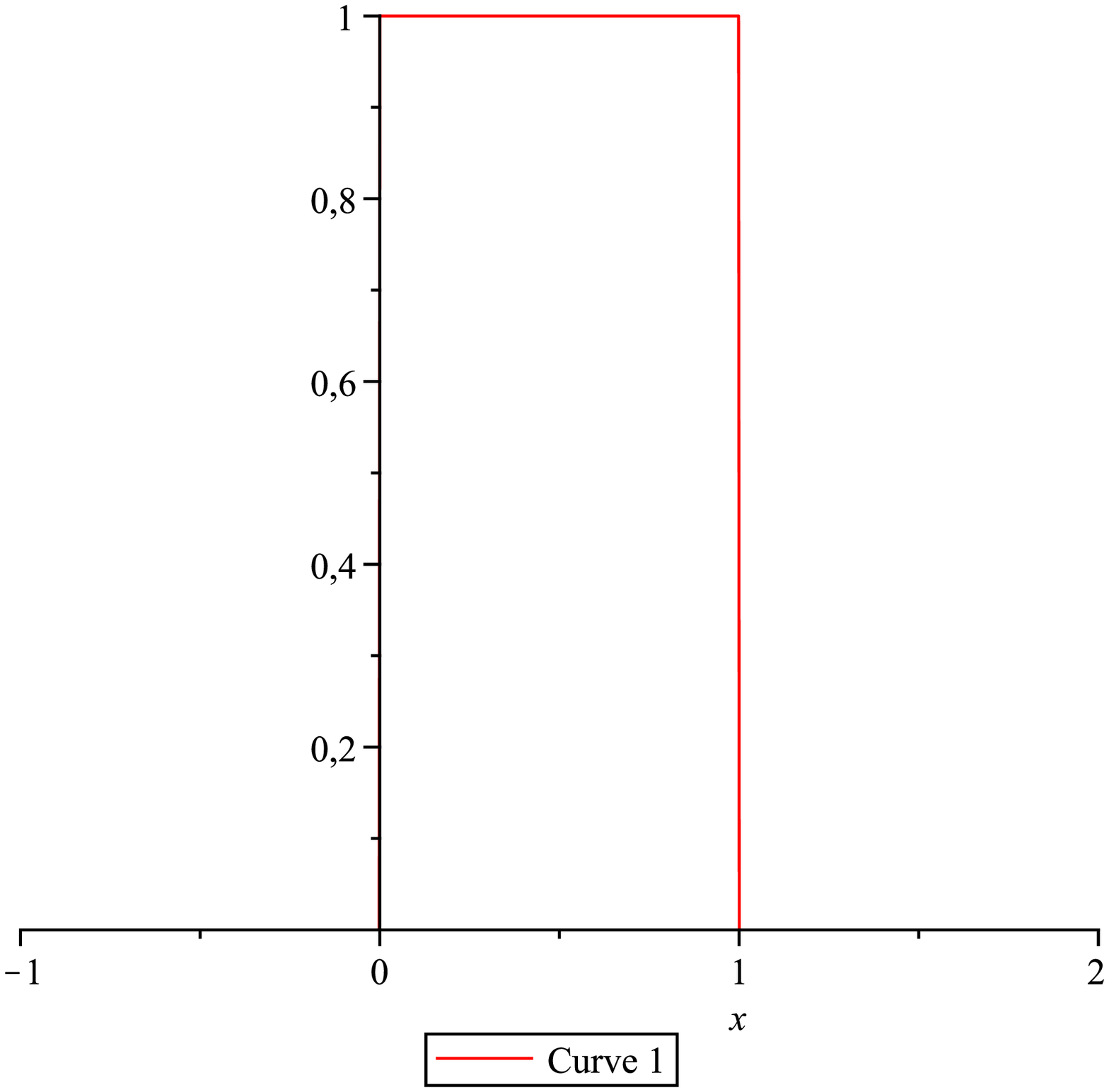}
 \includegraphics[width=37mm]{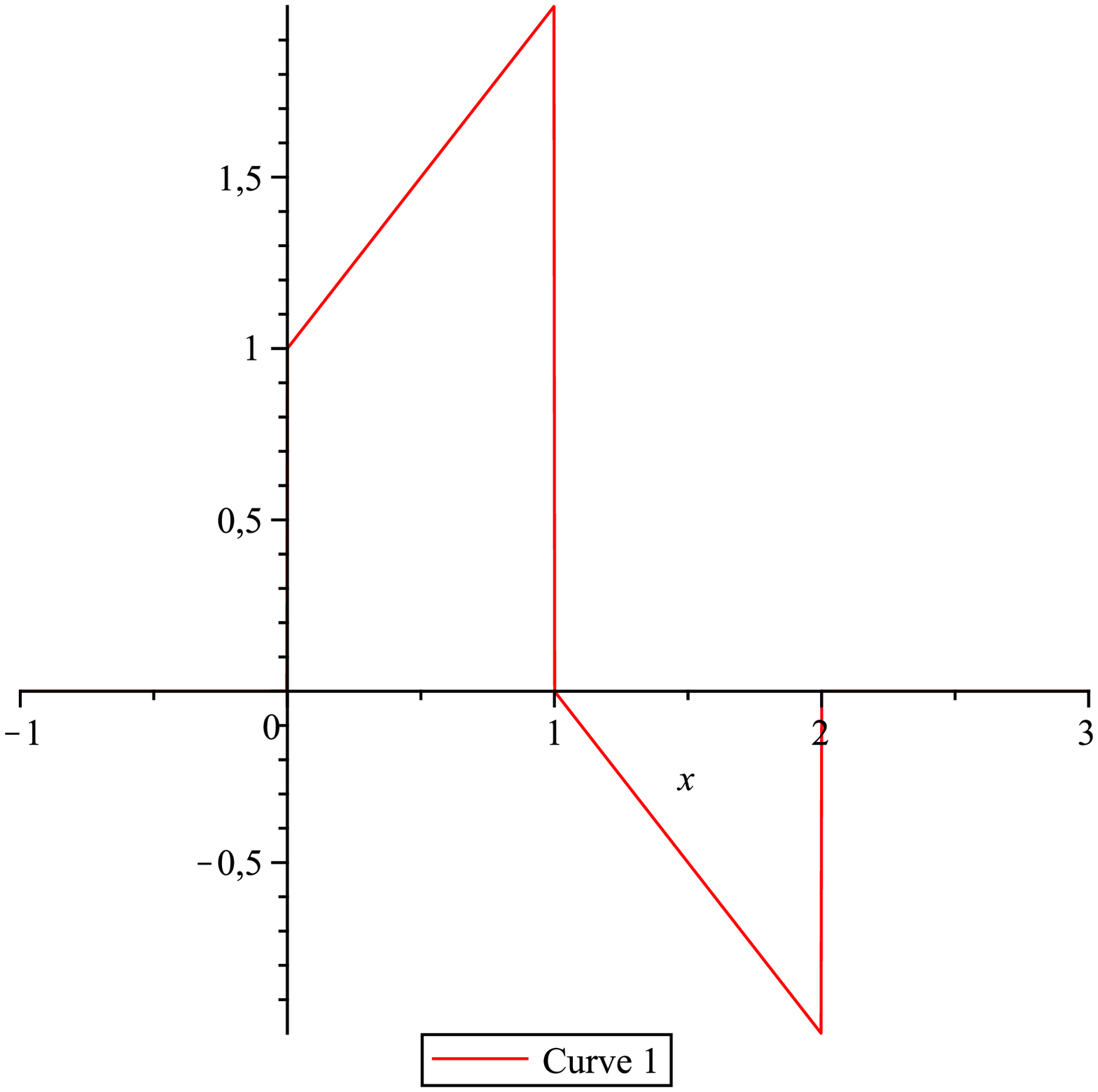}  \includegraphics[width=37mm]{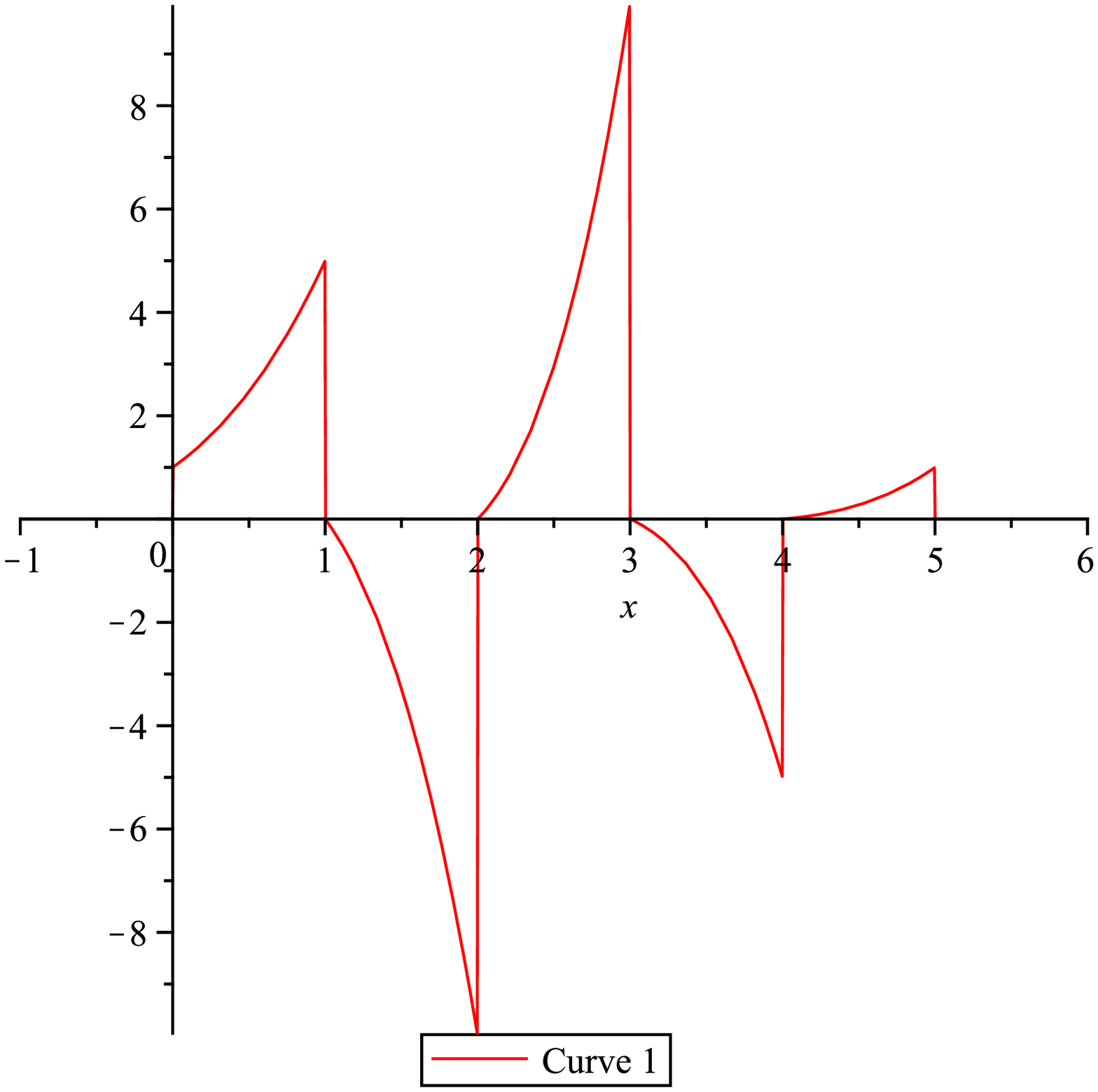}\\
 \caption{$Todd(X_1)_{pw} B_{X_1}, \quad Todd(X_2)_{pw} B_{X_2},\quad Todd(X_5)_{pw} B_{X_5}$}\label{tb1}
 \end{center}
\end{figure}

  \end{remark}

%
%The beauty of this theorem is that for interesting situations, the
%function $B_X*_d K$ has a geometric meaning, and is  piecewise polynomial with respect to large regions (set of regular values of the moment map,....). The inversion formula shows that $K$ has also a geometric meaning, and is given by polynomials in large regions.\marginpar{why polynomials?}
%

\bigskip

 We will give the proof  of this theorem in  Subsection \ref{proofinversion} after having introduced some further notions.

\subsection{Dahmen-Micchelli spaces}

Let us recall some facts on Dahmen-Micchelli polynomials.

If $I$ is a   sequence  of vectors, we define the operators  $\partial_I:=\prod_{a\in I}\partial_a$ and $\nabla_I:=\prod_{a\in I}\nabla_a$. These operators  are defined on distributions $B$ since  by duality we can set:
$$\langle t_a B, test\rangle =\langle B,t_{-a} test\rangle,  \hspace{1cm} \langle \partial_a B,test\rangle =-\langle B,\partial_a test\rangle .$$

If $Y$ is a subsequence of $X$,   by  $X\setminus Y$ we mean the complement  in $X$ of the sequence $Y$. If $S$ is a subset of $V$, we also employ  the notation    $X\setminus S$  for  the sequence of elements of $X$ not lying in $S$, and  $X\cap S$  for the sequence of elements of $X$  lying in $S$.
We have the following equality of distributions (cf. \cite{dp1} Proposition 7.14):

\begin{equation}\label{eq:nablabox}
\partial_Y B_X=\nabla_Y B_{X\setminus Y}.
\end{equation}

A subsequence $Y$ of $X$ will be called {\em long} if the sequence $X\setminus Y$ does not generate the vector space $V$.
A  long subsequence $Y$, minimal along the long subsequences, is also called a  {\em cocircuit}. In this case $Y=X\setminus H$ where $H$ is an admissible hyperplane.

In particular, if $Y=X\setminus H$, using equation (\ref{eq:nablabox}) we have that $\partial_Y B_X=\nabla_Y B_{X\cap H}$ is supported on the union of the  affine admissible hyperplanes which are translates of $H$ by elements of $Y$. So the restriction of $\partial_Y B_X$ to any alcove $\tau$ is equal to $0$.

Recall the definitions.

\begin{definition}

1)
The space $D(X)$ is the space  of (generalized) functions $B$ on $V$ such that
$\partial_Y B=0$ for all long subsequences $Y$ of $X$.

2)
The space $DM(X)$ is the space of integral valued functions $K$ on $\Lambda$ such that
$\nabla_Y K=0$ for all long subsequences $Y$ of $X$.

\end{definition}

Of course, it is sufficient to impose these equations for $Y$ running  along all cocircuits.

\medskip

  If $X$ spans $V$,            it is easy to see that $D(X)$ is a finite dimensional space  of polynomial functions on $V$
             and that $DM(X)$ is a free abelian group of finite rank of quasi-polynomial functions on $\Lambda$ \cite{DM1} (see \cite{dp1}).
             
             In our paper  we need to compare often $D(X)$ with $DM(X)$. In order to do this it is more convenient to extend $DM(X)$ to the spaces $DM(X)_{\Bbb R}:=DM(X)\otimes_{\Bbb R},$ $DM(X)_{\Bbb C}:=  DM(X)\otimes_{\Bbb C}$ respectively of real or complex valued functions satisfying the same difference equations as  $DM(X) $  (cf. Theorem \ref{theo:deltalambda}).  Sometimes by abuse of notations we shall drop the subscript $\Bbb C$  and  just write $DM(X)$ for $DM(X)_{\Bbb C} .$  Similarly we could take complex valued solutions of the diffrential equations getting the space  $D(X)_{\Bbb C}:=  D(X)\otimes_{\Bbb C}$.
             
             The restriction of a function $p\in D(X)$ (a polynomial) to $\Lambda$ is in $DM(X)\otimes_{\Bbb R}$.
If $X$ is a  unimodular system, this restriction map is an isomorphism.

The space $D(X)$ is invariant by differentiations. The space $DM(X)$ is invariant by translations by elements of $\Lambda.$

The following lemma follows from the definitions.

\begin{lemma}\label{lem:partialydx}
If $Y$ is a subsequence of $X$ such that $X\setminus Y$ still generates $V$, then
$\partial_Y D(X)$ is contained in $D(X\setminus Y)$
and $\nabla_Y DM(X)$  is contained in $DM(X\setminus Y)$.
\end{lemma}

\begin{remark}
In fact the operators $\partial_Y$ and $\nabla_Y$ are surjective onto  $D(X\setminus Y)$ and $DM(X\setminus Y)$ respectively. This is a more delicate statement, that we proved in \cite{dpv2} (and over $\Z$  for $DM$).
We will not use this  stronger statement here.
\end{remark}

If $\tau$ is an alcove contained  in $Z(X)$, then the polynomial $B_X^{\tau}$ is a non zero polynomial belonging to $D(X)$ (as seen from Equation (\ref{eq:nablabox})).

\begin{lemma}\label{daDMaD}
 If $K\in DM(X)$, then
 $B_X*_d K\in D(X).$

\end{lemma}

\begin{proof}

Let $H$ be an admissible hyperplane.
Then $$\partial_{X\setminus H} (B_X *_dK)=\partial_{X\setminus\underline H}B_X *_dK$$
$$=B_{X\cap H}*_d(\nabla_{X\setminus H }K)  =0. $$
 \end{proof}

The following result is proved in  \cite{DM1} (see also \cite{dp1}, \cite{myself}).

\begin{theorem}\label{theo:act}
If $K$ is the restriction to $\Lambda$ of a polynomial $k\in D(X)$, then  $B_X*_d K$ is equal to the polynomial
$I(X)k$.

Thus on the space $D(X)$, the operator  $B_X*' k:= B_X*_d K$, called semi--discrete convolution, is an isomorphism with inverse $Todd(X)$.

\end{theorem}

We will need some structure theory on $DM(X)$.

Let $\c$ be an alcove.
Let us consider any point $\epsilon\in \c$.
It is easy to see that the set $(\epsilon-Z(X))\cap \Lambda$ depends only of $\c$.
One  gives the following definition:

\begin{definition}
Let $\c$ be  an alcove. We denote by
$\delta(\mathfrak c\,|\,X):=(\epsilon-Z(X))\cap \Lambda$, where $\epsilon$ is any element of $\c$.
\end{definition}

We finally recall the following important  theorem of Dahmen-Micchelli \cite{DM1} ,\cite{DM3} (see \cite{dp1}).

 \begin{theorem}\label{theo:deltalambda}
 Let $\c$ be an alcove.
 For any $\xi\in \delta(\mathfrak c\,|\,X)$, there exists a unique Dahmen-Micchelli element $k_{\c}^{(\xi)} \in DM(X)$
 such that
$$k_\c^{(\xi)}(\xi)=1,$$
$$k_\c^{(\xi)}(\nu)=0$$
if $\nu\in \delta(\mathfrak c\,|\,X)$  and $\nu\neq \xi$.
 \end{theorem}

\subsection{Proof of the inversion formula}\label{proofinversion}

After these definitions, let us return to the proof of the inversion formula in the unimodular case.
We consider an alcove $\c$ contained in $Z(X)$ and containing $0$ in its closure.

\begin{proof} By Remark \ref{soeq}  it is enough to prove i).

 By definition,
$ \lim_{\c}(Todd(X)_{pw}B_X) (\lambda)=(Todd(X)B_X^{\lambda+\c})(\lambda)$.
If $(\lambda+\c)\cap Z(X)=\emptyset$, then  $B_X^{\lambda+\c}=0$ so
  $\lim_{\c} (Todd(X)_{pw} B_X)(\lambda)=0.$

  We now fix a point $\lambda$ such that  the alcove
  $\lambda+\c$  does  intersect $Z(X)$.
  The point $\lambda=0$ is such a point, by our assumption on $\c$.

  The condition $(\lambda+\c)\cap Z(X)\neq \emptyset$ is equivalent to the fact that $0\in  \delta(\lambda+\c\,|\,X) .$
  Remark that $\lambda=\lambda+\epsilon-\epsilon,\epsilon\in\c$ is also in $ \delta(\lambda+\c\,|\,X)$.
By Theorem \ref{theo:deltalambda} there is a unique element $p_{\lambda,\c}:=k_{\lambda+\c}^{(0)}$ in $DM(X)$  coinciding with $\delta_0$ on $ \delta(\lambda+\c\,|\,X)$.

Let us compute
  $(B_X*_d p_{\lambda,\c})(v )$ with
  $v\in \lambda+\c$. Using  the definitions, for such a $v$, we have
 $$(B_X*_d p_{\lambda,\c})(v )=\sum_{\nu\in \Lambda} p_{\lambda,\c}(\nu) B_X (v-\nu) =\sum_{\nu\in  \delta(\lambda+\c\,|\,X)} p_{\lambda,\c}(\nu) B_X (v-\nu).$$
   The second equality follows from the fact that the support of $B_X$ is $Z(X)$. As on $ \delta(\lambda+\c\,|\,X)$,  $p_{\lambda,\c}$ vanishes except at $0$, we obtain
from Lemma \ref{daDMaD} $$B_X^{\lambda+\c}=B_X*_d p_{\lambda,\c}.$$

 At this point we  use the fact that $X$ is a  unimodular system, so that the restriction map from $D(X)$ to $DM(X)$ is an isomorphism. Thus $p_{\lambda,\c}$ is the restriction to $\Lambda$ of a polynomial still denoted by $p_{\lambda,\c}$ belonging to $D(X)$ and, by Theorem \ref{theo:act},   $B_X^{\lambda+\c}=I(X)p_{\lambda,\c}$.
It follows that $Todd(X) B_X^{\lambda+\c}=p_{\lambda,\c}$ and
$$p_{\lambda,\c}(\lambda)= \lim_{\c}(Todd(X)_{pw} B_X)(\lambda).$$
As $p_{\lambda,\c}(\lambda)=0$, when $\lambda\neq 0$  and   $p_{\lambda,\c}(0)=1$, this proves our claim.
\end{proof}

 \subsection{ Inversion formula: the general case}
We keep the notations of \S \ref{Nop}.
  $G$  is a torus with group of characters of $\Lambda$.
  For $g\in G$ and $\lambda\in \Lambda$, define $$X^g:=\{a\in X \, | g^{a}=1\},\quad G_\lambda:=\{g\in G\,|\, g^\lambda=1\}.$$

For each $a\in X$  the set $G_a$ is a subgroup of codimension 1, these  groups generate a {\em toric arrangement}  $\mathcal A_X$, formed by all connected components of the intersections  of these groups $G_a$. Of particular importance are the vertices of the arrangement which can also be described as follows.

\begin{definition}\label{def:toricv}
  We say that a point $g\in G$ is a {\em toric vertex} of the  arrangement $\mathcal A_X$ if $X^g$  generates $V$.
                  We denote by $\mathcal V(X)\subset G$ the set of toric vertices of the arrangement $\mathcal A_X$.
  \end{definition}

   If $g$ is a vertex,  there is a basis $\sigma$ of $V$ extracted from $X$ such that $g^{a}=1$, for all $a\in \sigma$. We thus see that   the set $\mathcal V(X)$ is finite.
We also see that, if $X$ is unimodular, then $\mathcal V(X)$ is reduced to $g=1$.

For $g\in G$, we think of $g^\lambda\in \mathcal C[\Lambda]$ and denote by $\hat g$ the  operator on $\mathcal C[\Lambda]$, given by  multiplication  by $g^\lambda$: $(\hat g K)(\lambda)=g^{\lambda} K(\lambda)$. If $\nu\in \Lambda$, then
$\hat g t_\nu \hat g^{-1}= g^{\nu} t_\nu$.

%
%\begin{remark}\marginpar{I think this is useless}
%Assume that the series $\Theta(t)=\sum_{\lambda} K(\lambda) t^{\lambda}$ converges as a generalized function on $G$.
%Then $\Theta(gt)=\sum_{\lambda} (\hat g K)(\lambda) t^{\lambda}$.
%\end{remark}
%

\bigskip

We introduce next the {\em twisted}  difference and differential   operators.

 We set, for a vector $a $ or  for a sequence $Y$ of elements of $\Lambda$,
\begin{equation}\label{nabf}
\nabla_a^g:=1-g^{-a}t_a,\quad \nabla(g,Y)=\prod_{a\in Y}\nabla_a^g ,
\end{equation}

\begin{equation}\label{nabf2}
D_a^g:=1-g^{-a}e^{-\partial_a},\quad D(g,Y)=\prod_{a\in Y}D_a^g.
\end{equation}

The operator $\nabla(g,Y)$ acts on functions on $\Lambda$.
One has the formula
\begin{equation}\label{eq:comgnabla}
{ \hat g}^{-1} \nabla_Y \hat g=\nabla(g,Y).
\end{equation}

The operator $\nabla(g,Y)$ being a  linear combination of translation operators acts also on piecewise polynomial functions on $V$.
The operator $D(g,Y)$ acts on  piecewise polynomial functions on $V$ by its local action $D(g,Y)_{pw}$.

{\bf Be careful}:
The operators $D(g,Y)$ and $\nabla(g,Y)$ coincide on $\C[V]$, but their action is  not the same on  piecewise polynomial functions.
Indeed  the operator $f\mapsto D(g,Y)_{pw}f$  respects the support of $f$, while the operator $\nabla(g,Y) f$  may move the support of $f$.
\smallskip

If $g^a\neq 1$, then  $D_a^g=(1-g^{-a})+g^{-a}(1-e^{-\partial_a})$ is an invertible operator on polynomial functions
with inverse given by  the series of differential operators
$$(D_a^g)^{-1}= (1-g^{-a})^{-1}\sum_{k=0}^{\infty} (-1)^k (\frac{g^{-a}}{1-g^{-a}})^k (1-e^{-\partial_a})^k.$$

If $Y\subset X\setminus X^g$, we have also
  $D(g,Y)^{-1}=\prod_{a\in Y} (D_a^g)^{-1},$  an infinite series of differential operators.

\smallskip

For $K\in \mathcal C[\Lambda] $
and $g\in \mathcal V(X)$,  we define the  function
\begin{equation}\label{eq:defomega}
\omega_g(K):=B_{X^g}*_d \left(\hat g^{-1}\nabla_{X\setminus X^g}K\right).
\end{equation}

The function $\omega_g(K)$ is a piecewise polynomial function on $V$ with respect to $(X^g,\Lambda)$, thus a fortiori with respect to $(X,\Lambda)$.

\begin{theorem}\label{theo:main}

Let $\c$ be an alcove in $V$ containing $0$ in its closure and contained in $Z(X)$.
  Then
\begin{enumerate}[i)]

\item
 $\delta_0=\sum_{g\in \mathcal V(X)} \hat g \lim_{\c} \left(D(g, X\setminus X^g)^{-1} Todd(X^g)\right)_{pw}  (\nabla(g,X\setminus X^g)B_{X^g}).$

\item
For any $K\in \mathcal C[\Lambda]$,
  one has
the inversion formula:

$$K=\sum_{g\in \mathcal V(X)} \hat g \lim_{\c} \left(D(g, X\setminus X^g)^{-1} Todd(X^g)\right)_{pw}  \omega_g(K).$$

\end{enumerate}
\end{theorem}

\begin{remark}
The first assertion is a particular case of the second. Indeed, for $K=\delta_0$, as $\hat g\delta_0=\delta_0$, we have
  $$\omega_g(K)=B_{X^g}*_d(\hat g^{-1} \nabla_{X\setminus X^g}\hat g\delta_0)=B_{X^g}*_d\nabla(g,X\setminus X^g)\delta_0$$
  thus

  \begin{equation}\label{omegag}
  \omega_g(\delta_0)=\nabla(g,X\setminus X^g)B_{X^g}.
  \end{equation}

We will see later that the second assertion is a  consequence of the first.

\end{remark}

In order to prove Theorem \ref{theo:main}, we shall follow essentially the same method of proof used in the unimodular case.

 Notice that the restriction $K$  of a function $k\in D(X)$ to $\Lambda$ is an element of $DM(X)$.
Recall the following structure theorem.

\begin{theorem}(see \cite{dp1} Formula 16.1 and Theorem 17.15)\label{convver}
\begin{enumerate}

\item
   If $g\in \mathcal V(X)$ is a  toric vertex of the arrangement $X$ and $k\in D(X^g)$, the function
  $ \hat g K$  belongs to  $DM(X)$.\\

  \item
We have $DM(X)_{\Bbb C}=\oplus_{g\in \mathcal V(X)} \hat g D(X^g)_{\Bbb C}.$
\\

  \item Let $E(X)=\oplus_{g\neq 1}  \hat g D(X^g).$
Then, for any  $K\in E(X)$,
$B_X*_d K=0.$\\
\end{enumerate}

\end{theorem}

Given an element $g\in \mathcal V(X)$,  recall that  the map $\nabla_{X\setminus X^g}$ sends $DM(X )$ to  $DM(X^{g })$.
We have
\begin{lemma}\label{fornn} Take $g,h\in \mathcal V(X)$.
\begin{enumerate}[i)]
\item If $h\in \mathcal V(X^g)$, then
 \begin{equation}
\nabla_{X\setminus X^g}\hat h D(X^{h })\subset \hat h  D(X^{h}\cap X^g ).\end{equation}

\item If $h\notin \mathcal V(X^g)$, then $$\nabla_{X\setminus X^g}\hat h D(X^{h})=0.$$

\end{enumerate}
\end{lemma}

\begin{proof}
Let $K$  be a function on $\Lambda$.
An operator $\nabla_Y$  acting on $\hat hK$ can be analyzed by decomposing $Y=Z\cup R$ into the part $Z$ of elements $a\in Y$  such that $h^a=1$ and the complement $R$. Then
\begin{equation}
\label{bbl}\nabla_Y \hat h K=\hat h (\hat h^{-1} \nabla_Z \hat h)  (\hat h^{-1} \nabla_R \hat h)  K=\hat h \nabla_Z \nabla(h,R)K.
\end{equation}
In particular we apply this to  $ Y={X\setminus X^g}$ which, for a point $h$ of the arrangement,  we separate into   the  subsequences $Z:=X^h\cap (X\setminus X^g) =X^h\setminus X^g $ and  $ R:=X\setminus (X^h\cup X^g)$. We obtain from \eqref{bbl} $$\nabla_{X\setminus X^g}\hat h D(X^{h })=\hat h  \nabla_{X^h\setminus X^g }\nabla(h, R) D(X^{h }).$$ The operator  $\nabla(h, R)$, a finite combination of translations, preserves the space $D(X^h)$. Thus we get \begin{equation}
\nabla_{X\setminus X^g}\hat h D(X^{h })\subset \hat h  \nabla_{X^h\setminus X^g } D(X^{h })= \hat h  \nabla_{X^h\setminus  X^g } D(X^{h }).\end{equation}
and {\it i)}  follows from Lemma \ref{lem:partialydx} .

Furthermore,
by definition, a point $h$ is a vertex of  the arrangement $X^g$ if and only if  the vectors in $X^h\cap X^g$ span $V$.
So if $h\notin \mathcal V(X^g)$,    $X^h\setminus X^g $  is a long subsequence of $X^h$ and
$\nabla_{X^h\setminus X^g } D(X^{h })=0$  getting  {\it ii)}.

\end{proof}

\begin{proposition}
If $K\in DM(X)$, then
$$\omega_g(K)=B_{X^g}*_d \left(\hat g^{-1}\nabla_{X\setminus X^g}K\right)\in D(X^g).$$\end{proposition}
\begin{proof}
Indeed, $\nabla_{X\setminus X^g}K\in DM(X^g)$ and $\hat g^{-1}$ preserves $DM(X^g)$. Thus $\omega_g(K)$ is a polynomial belonging to $D(X^g)$ by Lemma \ref{daDMaD}.

\end{proof}

\begin{proposition}\label{fondi} Let $K\in DM(X)$. Write
   $K=\sum_ {g\in \mathcal V(X)} \hat g K_g$, with $k_g\in D(X^g)_{\mathbb C}$ restricting to $K_g$. Then  we have
 $$k_g=D(g,X\setminus X^g)^{-1} Todd(X^g)\omega_g(K).$$

 \end{proposition}

 \begin{proof}
 Let    $k_h\in D(X^h)_{\mathbb C}$.
Let us compute $\omega_g(\hat h K_h)$ for each $g\in \mathcal V(X)$.

By Lemma \ref {fornn},
$\nabla_{X\setminus X^g} \hat h K_h$ is zero unless $h$ is a vertex of $X^g$.

Assume now that $h$ is a vertex of $X^g$.  Then $\nabla_{X\setminus X^g} \hat h K_h=\hat h Z$ where  $Z$ is the restriction of a polynomial $z$ lying in $D(X^g\cap X^h)_{\mathbb C}$.

   Clearly $g^{-1}h$ is also a vertex of $X^g$ and $X^g\cap X^{g^{-1}h} =X^g\cap X^h$. We deduce using Lemma \ref{fornn} i) that, if $ g\neq  h$, $$\hat g^{-1}\nabla_{X\setminus X^g} \hat h K_h\in \widehat{g^{-1}h}D(X^g\cap X^{g^{-1}h})_{\mathbb C}\subset E(X^g).$$ So, by Theorem \ref{convver}   $$\omega_g(\hat h K_h)=B_{X^g}*_d\left(\hat g^{-1}\nabla_{X\setminus X^g}\hat h K_h\right )=0.$$

 Finally, if $h=g$,    we obtain that $\hat g^{-1}\nabla_{X\setminus X^g} \hat g K_g$ is the restriction to $\Lambda$ of the polynomial $ D(g,X\setminus X^g)k_g\in
 D(X^g)_{\mathbb C}$.
  By Theorem \ref{theo:act}, the semi-discrete convolution acts by the operator $I(X^g)$ on $D(X^g)$, so that we get
$$\omega_g(\hat g K_g)=D(g,X\setminus X^g)I(X^g) k_g.$$
In conclusion,
   $$\omega_g(\hat h K_h)=\begin{cases}0\ \ \text{if\ } h\neq g\\ D(g,X\setminus X^g)I(X^g) k_g \ \ \text{if\ } h=g.\end{cases}$$
This   implies our claims.
 \end{proof}

 We are now ready to prove our main  Theorem \ref{theo:main}.

   We compute the function $j$ on $\Lambda$ given by
 $$j=\sum_{g\in \mathcal V(X)} \hat g \lim_{\c} \left(D(g, X\setminus X^g)^{-1} Todd(X^g)\right)_{pw}  (\nabla(g,X\setminus X^g)B_{X^g}).$$

  We proceed as in the proof of Theorem \ref{theo:inversionuni}.
  The support of $ \nabla(g,X\setminus X^g)B_{X^g} $ is contained in $Z(X)$. Indeed if $I$ is a subsequence of  $X\setminus X^g$, $\sum_{i\in I}a_i+Z(X^g)\subset Z(X)$.
 So, if    $\lambda\in \Lambda$ and $(\lambda+\c)\cap Z(X)$ is empty, we see that $j(\lambda)=0$.

Now assume that $(\lambda+\c)\cap Z(X)$ is not empty.
Then the points $0$ and $\lambda$ belong to $  \delta(\lambda+\c\,|\,X)$.
 Let
$p_{\lambda,\c}$ be the element of $DM(X)$ which coincides with
$\delta_0$ on $ \delta(\lambda+\c\,|\,X)$.
Let us show  that
 $j(\lambda)=p_{\lambda,\c}(\lambda)$, so that we will obtain i).

We decompose $p_{\lambda,\c}$ according to Proposition \ref{fondi}. It is sufficient to prove
 that the polynomial $\omega_g(p_{\lambda,\c})(v)$ coincides with the piecewise polynomial function   $\nabla(g,X\setminus X^g)B_{X^g}$ on $\lambda+\c$.  That is, if $v\in \lambda+\c$,
\begin{equation}\label{identi}\omega_g(p_{\lambda,\c})(v)=
(\nabla(g,X\setminus X^g)B_{X^g})(v).
\end{equation}
Indeed, by Proposition \ref {fondi}, we have
$$\omega_g(p_{\lambda,\c})=B_{X^g}*_d \hat g^{-1} \nabla(X\setminus X^g) p_{\lambda,\c}=B_{X^g}*_d  \nabla(g,X\setminus X^g) \hat g^{-1} p_{\lambda,\c}$$ so, since semi-discrete convolution commutes with translation, $$\omega_g(p_{\lambda,\c})= \nabla(g,X\setminus X^g)( B_{X^g}*_d \hat g^{-1} p_{\lambda,\c}).$$

For any   subsequence  $I$   of $X\setminus X^g$, set  $a_I=\sum_{i\in I} a_i$. In order to see \eqref{identi}, we need to show  that for $v\in \lambda+\c$,
\begin{equation}\label{identi34}B_{X^g}( v-a_I)=( B_{X^g}*_d \hat g^{-1}p_{\lambda,\c})( v-a_I).\end{equation}
By definition,  the right hand side of this  expression equals
$$\sum_{\nu\in \Lambda}   g^{-\nu}p_{\lambda,\c}(\nu) B_{X^g}(v-a_I-\nu).$$
Now the summand  $B_{X^g}(v-a_I-\nu)$ is zero except if
$v-a_I-\nu$ is in the zonotope $Z(X^g)$. But if this is the case, $v-\nu\in Z(X^g)+a_I\subset Z(X)$, so necessarily  $\nu$ lies in $ \delta(\lambda+\c\,|\,X)$. From this
\eqref{identi34} and hence \eqref{identi} follow by the special choice of   $p_{\lambda,\c}$.
The proof of the first item of Theorem \ref{theo:main} is finished.

Let us prove the second item.  Define
 $$j(K)=\sum_{g\in \mathcal V(X)} \hat g \lim_{\c} \left(D(g, X\setminus X^g)^{-1} Todd(X^g)\right)_{pw}
B_{X^g}*\hat g^{-1}  \nabla_{X\setminus X^g} K$$
$$=\sum_{g\in \mathcal V(X)} \hat g \lim_{\c} \left(D(g,X\setminus X^g)^{-1} Todd(X^g)\right)_{pw}
  B_{X^g}*\nabla(g,X\setminus X^g) \hat g^{-1} K.$$

Then $K\to j(K)$ is an operator of the form $\sum_{g} \hat g R_g \hat g^{-1} K$ where $R_g$ is an operator commuting with translations by elements of $\Lambda$.
Thus the operator $K\to j(K)$ commutes with translation.
Furthermore it is clear that the formula for $j(K)(\lambda)$ involves only a finite number of values of $K(\nu)$ (contained in $\lambda-Z(X)$).
Thus to prove that $j(K)=K$, it is sufficient to prove it for $K$ with finite support.  By translation invariance this case follows from the formula for $\delta_0$.

\section{Partition functions and splines}\label{bv}

\subsection{The formula of Brion-Vergne}

If $Y$ is a sequence of elements of $\Lambda$  generating a pointed cone $\Cone(Y)$, then we can define  the series
$$\Theta_Y=\prod_{a\in Y}\sum_{k=0}^\infty e^{ka}.$$

We write $$\Theta_Y=\sum_{\lambda\in \Lambda}{\mathcal P }_Y
(\lambda)e^{\lambda}$$ where  ${\mathcal P }_Y \in {\mathcal C}_\Z [\Lambda]$ is, by definition, the {\em partition
function} associated to $Y$.
For any subsequence $S$ of $Y$, we then have
\begin{equation}\label{eqnablaY}
\nabla_S{\mathcal P }_Y ={\mathcal P }_{Y\setminus S}.
\end{equation}

In particular $\nabla_Y \mathcal P_Y=\delta_0$, the delta function on $\Lambda$.

Similarly, we can define the  multivariate spline $T_Y$
 which is the tempered distribution  on $V$
defined by:
\begin{equation}\label{multiva}
\langle T_Y\,|\,f\rangle = \int_0^\infty\dots\int_0^\infty
f(\sum_{i=1}^kt_i a_i)dt_1\dots dt_k
\end{equation}
where $Y=[a_1,a_2,\ldots, a_k].$

For any subsequence $S$ of $Y$, we then have
\begin{equation}\label{eqpartialY}
\partial_S T_Y =T_{Y\setminus S}.
\end{equation}

In particular $\partial_Y T_Y=\delta_0$, the $\delta$-distribution on $V$.

    Decomposing a  ray as a sum of intervals, the following formula of Dahmen-Micchelli  follows.

\begin{proposition}\label{dapat}
$$T_Y(v)=(B_Y*_d{\mathcal   P}_Y)(v)=\sum_{\lambda\in\Lambda}B_Y(v-\lambda){\mathcal   P}_Y(\lambda).$$
\end{proposition}

  Let us see that  the formulae obtained in  \cite{brionvergne} are a corollary of the general inversion formula  of Theorem \ref{theo:main}.
 We assume that $X$ generates the vector space $V$ and spans a pointed cone (thus with non empty interior).
 It follows that  $T_X$ is a piecewise polynomial distribution on $V$.

 Let us apply the inversion formula of Theorem \ref{theo:main}  to the partition function $\mathcal P_X$.
 From the equations
 $$    \hat g^{-1}\nabla_{X\setminus X^g}\mathcal P_X=\hat g^{-1}\mathcal P_{X^g}=\mathcal P_{X^g},\hspace{1cm}  B_{X^g}*_d\mathcal P_{X^g}   =T_{X^g}, $$
 we deduce from Theorem \ref{theo:main}
\begin{theorem}

Let $\c$ be   an alcove contained in  $\Cone(X)$ and having 0 in its closure, then
  $\mathcal P_X$ coincides with:
\begin{equation}\label{bbvv}
\sum_{g\in \CV(X)} \hat g \lim_\c
\left(\prod_{b\in X^g}\frac{  \partial_b} {
1-e^{-\partial_b} } \prod_{a\in X\setminus X^g}\frac{1}{1-g^{-a} e^{-\partial_a}} T_{ X^g}\right).
\end{equation}

\end{theorem}

\bigskip

More generally, given a sequence of non zero vectors $Y$ in $\Lambda$, where we do not  necessarily assume  that $Y$ spans a pointed cone, we can define
{\em polarized partition functions} as follows.
Consider the open subset $\{u\in V^*\,|\,\ll u,a\rr \neq 0 {\, \rm for \, all\,} a\in Y\}$ of $V^*$.  A connected component $F$ of this open set will be called a  {\em  regular face} for $Y$. An element $\phi\in F$ decomposes $Y=A\cup B$ where $\phi$ is positive on $A$ and negative on $B$. This decomposition depends only upon $F$. We define

$$\Theta_Y^F=(-1)^{|B|}\prod_{a\in A}\sum_{k=0}^\infty e^{ka}
\prod_{b\in B}\sum_{k=1}^\infty e^{-kb}
.$$

Thus   $\Theta_Y^F=(-1)^{|B|} e^{-\sum_{b\in B}b} \Theta_{A\cup -B}.$

Morally, $\Theta_{Y}^{F}=\prod_{a\in Y}\frac{1}{1-e^a}=\prod_{a\in A}\frac{1}{1-e^a}\prod_{b\in
B}\frac{-e^{-b}}{1-e^{-b}}.$ But we  need to reverse the sign of
 the vectors in $B$ in order to insure that they lie in a given pointed cone (depending on $F$) so that the
convolution product of the corresponding geometric series makes
sense.

We write $$\Theta_Y^F=\sum_{\lambda\in \Lambda}{\mathcal P }_Y^F
(\lambda)e^{\lambda}$$ where  ${\mathcal P }_Y^F \in {\mathcal C} [\Lambda]$ is the {\em polarized (by $F$) partition
function}. The polarized partition function is a $\Z$-valued function on $\Lambda$.

\

We define similarly
$$T_Y^F=(-1)^{|B|}T_{A,-B}$$
a distribution on $V$ and
call it the polarized multispline function.

As for Proposition \ref{dapat} it is easy to verify the following,
 \begin{proposition}\label{lem:BZ}
  For  any regular face $F$ for $Y$, one has
 $B_Y*_d\mathcal P_Y^F=T_Y^F$
 \end{proposition}

 Proposition \ref{lem:BZ} implies that $B_X=\nabla_X T_X^F$. Thus $B_X$ is a linear combination of translates of multisplines $T_X^F$.

 \subsection{The spaces ${\mathcal F}(X)$ and $\mathcal G(X)$}

In this subsection, we   recall the definitions of
the subspaces ${\mathcal F}(X)$, a subspace of functions on $\Lambda$, and $\mathcal G(X)$  a subspace of distributions on $V$ introduced in \cite{dpv1}.
 They will be central objects in Part II  as these spaces are related to  the
equivariant  $K$-theory and cohomology of some $G$-spaces.

We use the convolution sign $*$ for convolutions between functions on $\Lambda$, or distributions on $V$, or semi discrete convolution between a function on $\Lambda$ and distributions on $V$. The meaning will be clear in the context.

\begin{definition}
A subspace  $\underline s$ of $V$ is called {\it rational} (relative to $X$) if  $\underline s$ is the vector space generated by  $X\cap \underline s$.

We shall denote by $\mathcal S_X$ the set of rational subspaces.

\end{definition}

Denote by $\mathcal C_\Z[\Lambda]$ the $\Z$-module  of $\Z$-valued functions on $\Lambda$.
Define  the following subspace of $\mathcal C_\Z[\Lambda]$.

\begin{definition}\label{graf}

$$\mathcal F(X):=\{f:\Lambda\to \mathbb Z\,|\,  \nabla_{X\setminus  \underline r }
f \text{ is supported on }  \Lambda\cap \underline r \text{ for all  } \underline r\in \mathcal S_X  \}.
$$
We set $\tilde{\mathcal F}(X)$ to be the $\Z[\Lambda]$ module generated by $\mathcal F(X)$.
\end{definition}
The space  $\mathcal F(X)$  contains clearly the space
of Dahmen-Micchelli quasi polynomials $DM(X)$ and all polarized partition functions $\mathcal P_X^F$.

%\marginpar{I added a remark on continuity}

\begin{remark}
Assume that $X$ is unimodular. Let $\tau$ be a tope. We have proven in \cite{dpv1} that a function $f$ in $\mathcal F(X)$ coincide on $(\tau-Z(X))\cap \Lambda$ with the restriction  of a polynomial function $h^{\tau}$.
Thus we see that this collection of functions $h^{\tau}|_{\tau}$ extends to a continuous function on  the cone $Cone(X)$ generated by $X$.
\end{remark}

 We have a precise description of $\mathcal F(X)$ and  $\tilde{\mathcal F}(X)$  in Theorem 4.5 of  \cite{dpv1}.

\begin{theorem}\label{geneF}

\begin{itemize}
\item
Choose,  for every rational space $\underline r$,
  a  regular face $F_{\underline r}$ for  $X\setminus \underline r$. Then:
\begin{equation}\label{decomp}
\mathcal F(X)=\oplus_{\underline r\in S_X} {\mathcal
P}_{X\setminus\underline r}^{F_{\underline r}}*DM(X\cap \underline r ).
\end{equation}
 \item $\tilde{\mathcal F}(X)$
   is spanned over $\Z[\Lambda]$ by the elements $\mathcal P_X^F$ as $F$ runs over all regular faces for $X$.
\end{itemize}
\end{theorem}

Thus any Dahmen-Micchelli  quasi-polynomial $p$  in $DM(X)$  can be written as a linear combination of polarized partition  functions $\mathcal P_X^F$, for various $F$. It is not so easy to do it explicitly.

The easiest instance  of this result is when $\Lambda=\mathbb Z$. In this case one has only two rational spaces $\mathbb R$ and $\{0\}$. Thus our results says that, if we take as $F_{\{0\}}$ the positive half-line $\mathbb R^{\geq 0}$,  every element in $\mathcal F(X)$ can be written uniquely as the sum of a quasi polynomial in $DM(X)$ and a multiple of the partition function ${\mathcal
P}_X ^{\mathbb R^{\geq 0}}$. 

In general the decomposition \eqref{decomp} is not so easy to compute  explicitly.

\bigskip

We defined  in \cite{dpv1} an analogue space of piecewise polynomial distributions on $V$.

Denote by $\mathcal D'(V)$ the space of  distributions on $V$.
Let $\underline r$ be a vector subspace in $V$. We have an embedding $j:\mathcal D'(\underline r)\to \mathcal D'(V)$ by $\ll j(\theta),f\rr =\ll \theta,f|\underline r\rr $ for any $\theta\in \mathcal D'(\underline r)$, $f$ a test function on $V$. We denote the image $j(\mathcal D'(\underline r))$ by $\mathcal D'(V, \underline r) $ (sometimes we even identify $ \mathcal D'(\underline r) $ with $\mathcal D'(V, \underline r) $  if there is no ambiguity).
 We next define
 the vector space:

 \begin{definition}\label{gra}
$$ \mathcal G(X):=
\{f\in {\mathcal D'}(V)\,|\,
\partial_{X\setminus\underline r}f\in \mathcal D'(V, \underline r ),
\text{ for all } \underline r\in \mathcal S_X\}.
$$
We set $\tilde{ \mathcal G}(X) $ to be the module generated by $ \mathcal G(X)$ under the action of the algebra $S[V]$ of differential operators with constant coefficients. \end{definition}

 It is clear that $\mathcal G(X)$ contains the space $D(X)$ of Dahmen-Micchelli polynomials (we identify freely a locally
 $L^1$-function  $p$ and the distribution $p(v)dv$ using our choice of Lebesgue measure) as well as the polarized multisplines $T_X^F$.

\begin{theorem}\label{gestad1}

\begin{itemize}

\item
 Choose,  for every rational space $\underline r$,
  a  regular face $F_{\underline r}$ for  $X\setminus \underline r$. Then:
\begin{equation}\label{decomp1}
\mathcal G(X)=\oplus_{\underline r\in S_X}
T_{X\setminus r}^{F_{\underline r}}*D(X\cap \underline r ).
\end{equation}
\\

\item
The space $\tilde{ \mathcal G}(X) $ is  generated as $S[V]$ module by  the distributions $T_X^F$ as $F$ runs over all regular faces for $X$.

\end{itemize}
\end{theorem}

It follows from this theorem  that any $\theta$ in
$\mathcal G(X)$  is a piecewise polynomial distribution on $V$.

\subsection{An isomorphism}  We want to show now the strict relationship between the two spaces   ${\mathcal F(X)}$ and ${\mathcal G(X)}$.   We may use real valued functions $DM(X)\otimes\mathbb R$ and ${\mathcal F(X)}\otimes\mathbb R$ defined by the same difference equations.

The spaces   ${\mathcal F(X)}$ and ${\mathcal G(X)}$ are related by the semi-discrete convolution with the Box spline $B_X$.  Indeed, the following lemma generalizes the fact that a Dahmen-Micchelli quasi polynomial becomes a polynomial in $D(X)$ by  the semi-discrete convolution.

\begin{lemma}
If $f\in {\mathcal F(X)}$, then  $f*B_X\in {\mathcal G(X)}$.
\end{lemma}
\begin{proof}
If $f\in {\mathcal F(X)}$, then  $\nabla_{X\setminus  \underline r }
f $ is supported on $  \underline r $ for every
 rational subspace    $ \underline r $.  We need to show that  $\partial_{X\setminus\underline r}f*B_X\in \mathcal D( \underline r )$ for every
 rational subspace    $ \underline r $.

 We have from Formula \eqref{eq:nablabox}
 $$\partial_{X\setminus\underline r}f *B_X=f *\partial_{X\setminus\underline r}B_X=f *\nabla_{X\setminus  \underline r }B_{X\cap\underline r} =(\nabla_{X\setminus  \underline r }f )*B_{X\cap\underline r} . $$ Since  $\nabla_{X\setminus  \underline r }
f $ is supported on  $ \Lambda\cap \underline r $, we have that  $(\nabla_{X\setminus  \underline r }f )*B_{X\cap\underline r}\in \mathcal D( \underline r )$  as desired.
 \end{proof}

 \begin{theorem}\label{lisom}
 The map $f\mapsto f*B_X$ induces a surjective map
$$i:{\mathcal F(X)}\otimes_{\mathbb Z}\mathbb R\to {\mathcal G(X)} $$ compatible with the two decompositions  \eqref{decomp} and  \eqref{decomp1}.\smallskip

If $X$ is unimodular, $i$ is a linear isomorphism.

\end{theorem}

\begin{proof}
We have
$${\mathcal G(X)}=\oplus_{\underline r\in S_X}
T_{X\setminus r}^{F_{\underline r}}*D(X\cap \underline r ),$$
while
 $$
{\mathcal F(X)}=\oplus_{\underline r\in S_X} {\mathcal
P}_{X\setminus\underline r}^{F_{\underline r}}*DM(X\cap \underline r ).$$

 Under the mapping $f\to f*B_X$,  we have that $DM(X)\otimes\mathbb R$ maps surjectively to $D(X)$. Furthermore, in the  unimodular case, it induces a linear isomorphism onto $D(X)$.

Then consider an  element   ${\mathcal
P}_{X\setminus\underline r}^{F_{\underline r}}*u,\ u\in DM(X\cap \underline r )$ we have
$$ {\mathcal
P}_{X\setminus\underline r}^{F_{\underline r}}*u*B_X={\mathcal
P}_{X\setminus\underline r}^{F_{\underline r}}*u*B_{X\setminus\underline r}*B_{X\cap\underline r}=T^{F_{\underline r}}_{X\setminus\underline r}*(B_{X\cap\underline r}*u)$$ and the claim follows.
\end{proof}

%\marginpar{I added some thing on continuity}

\begin{remark}\label{cont}
If $X$ is unimodular,  the inverse of $i$ is given by the deconvolution formula:
$$i^{-1}h=\lim_{\c} Todd(X)*_{pw}h$$
where $\c$ is an alcove contained in $Z(X)$.

Consider a function $h\in {\mathcal G(X)} $. Then the locally polynomial function
$Todd(X)*_{pw}h$ coincide on topes $\tau\cap \Lambda$ with the function $g^{\tau}$ with $g=i^{-1}(h)\in \mathcal F(X)$.
It follows from the continuity properties of elements of $\mathcal F(X)$ that the locally polynomial function $Todd(X)*_{pw}h$ extends continuously on the cone generated by $X$.
In particular, if $Z(X)$ contains $0$ in its interior,  $Todd(X)*_{pw}h$ extends continuously on $V$.

The Box spline $B_X$ is a combination of translates of elements $T_X^F$ which belongs to ${\mathcal G(X)} $. It follows that $Todd(X)_{pw}B_X$ extends  continuously on $V$.
 \end{remark}

\newpage

\part{Geometry}

\section{ Equivariant $K$-theory and equivariant cohomology}

\subsection{Preliminaries} Although the theory we are going to review exists for a general compact Lie group $G$, we restrict our treatment to the case in which $G$ is a compact torus and denote by $\Lambda$ its character group.

We want to apply the preceding  purely algebraic results in order to compare the infinitesimal index, and the index associated to the symbol of a transversally elliptic operator on a linear representation   of $G$.

Let $N$ be a $G$ manifold provided with a real $G$-invariant form
$\sigma$.  We assume that $N$ is oriented.
 If $v_x$ is the vector field on $N$ associated to $x\in \mathfrak g$,  the moment
map $\mu: N\to \mathfrak g^*$ is defined by
$\ll \mu(n),x\rr =-\langle\sigma, v_x\rangle$, ($ n\in N,x\in \g$  and
the sign convention is that $v_x=\frac{d}{d\epsilon} \exp(-\epsilon x)n|_{\epsilon=0}$).
Define $Z$ as the zero fiber of the moment map.

%In \cite{dpv3}, we have   defined an {\em infinitesimal index} called $infdex$ which a homomorphism of $S[\g^*]$ modules, from the $H^*_{G,c}(Z)$ (the equivariant cohomology with compact support) to the space $D(\g^*)$  of  distributions on $\g^*$.

              \begin{remark}\label{rem:liouville}

We will mainly apply the construction described below to the case  where $N=T^*M$ is the cotangent bundle to a $G$-manifold $M$,
and $\omega$ is the Liouville one form defined for $m\in M$, $\xi\in T_m^*M$ and $V$ a tangent vector at the point $(m,\xi)\in T^*M$  by

 \begin{equation}\label{eq:liouville}
\omega_{m,\xi}(V)=\ll \xi,p_*V\rr.
                                 \end{equation}

By definition, the symplectic form $\Omega=-d\omega$ is the symplectic  form of $T^*M$, and we will use the corresponding orientation of $T^*M$ to compute integrals of differential forms on $T^*M$.

If $v_x$ is the
vector field  on $M$ associated to $x\in \mathfrak g$, the moment
map $\mu: T^*M\to \mathfrak g^*$ is then
$\ll \mu(m,\xi),x\rr =-\langle\xi, v_x\rangle$,  and
the zero fiber $Z$ of the moment map is denoted by $T^*_GM$. In a point $m\in M$ the fiber of the projection $p:T_G^*M\to M$
is the space of covectors conormal to the $G$ orbit through $m$.

For reasons explained later, we will use the opposite form $-\omega$, the opposite moment map $-\mu$ and, as stated before, the orientation given by $-d\omega$.

\end{remark}
\subsection{Transversally elliptic symbols and their index}
Let $Z$ be a topological space provided with an action of $G$.
Let $\CE^+$ and $\CE^-$ be two
complex equivariant vector bundles over $Z$.
Let  $\Sigma:\CE^+\to \CE^-$ be a $G$-equivariant morphism, and for $z\in Z$, denote by
 $\Sigma_z: \CE^+_z\to \CE^-_z$ the corresponding linear map.
 Recall that  {\em the support of $\Sigma$} is the subset of $Z$ consisting of elements $z$ where $\Sigma_z$ is not invertible.

Let us now recall  the definition of the multiplicity index map.

Let $\CE^+$ and $\CE^-$ be two
complex equivariant vector bundles over $M$.
 If $\Sigma:p^*\CE^+\to p^*\CE^-$ is a $G$-equivariant morphism,
 we say that $\Sigma$ is a $G$-equivariant  symbol.
Thus, for $m\in M,\xi\in T^*_mM$,  $\Sigma_{m,\xi}$ is a linear map : $\CE^+_m\to \CE^-_m$.
If the support of $\Sigma$ is a compact set, we say that $\Sigma$ is an elliptic symbol, and $\Sigma$ determines  an element $[\Sigma]$  in $K_G^0(T^*M)$.
If the support of $\Sigma$ intersects  $T^*_G M$ in a compact set, we say that $\Sigma$ is a transversally elliptic symbol (it is elliptic in the directions transverse to the $G$-orbits). Then   $\Sigma$ determines an element $[\Sigma]$  in $K_G^0(T^*_GM)$ and all elements of this group are obtained in this way (cf. \cite{At}).

Recall that  Atiyah-Singer \cite{At} have associated to any  transversally elliptic  symbol a   virtual representation of   $G$ of trace class. The induced trace of operators associated to smooth functions on $G$ is its index, a generalized function on $G$. By taking Fourier coefficients, one then gets a  homomorphism of $R(G)=\mathbb Z[\Lambda]$ modules
$${\rm ind}_m:K_G^0(T_G^*M)\to  \mathcal C_{\mathbb Z}[\Lambda]$$ called the index multiplicity function.
When the symbol $\Sigma$ is elliptic, one gets that ${\rm ind}_m(\Sigma)$ has finite support and the index is a virtual character of $G$.

 In \cite{dpv2}  we have studied in particular the case   $M=M_X$ the linear representation associated to a list $X$ of characters and proved that ${\rm ind}_m:K^0_G(T^*_GM_X)\to \mathcal C_\Z[\Lambda]$ gives an isomorphism of $\mathbb Z[\Lambda] $-modules onto $\tilde{\mathcal F}(X)$ (see Definition \ref{graf}). Moreover, if $M_X^f$  denotes the open set of points with finite stabilizer we have that the map ${\rm ind}_m$ establishes    an isomorphism of $K^0_G(T^*_GM_X^f)$ with $DM(X)$.

 Since, as we have recalled in Theorem \ref{geneF}, $\tilde{\mathcal F}(X)$
   is spanned over $\Z[\Lambda]$ by the elements $\mathcal P_X^F$ as $F$ runs over all regular faces for $X$, in order to find generators of
$K^0_G(T^*_GM_X)$, we are going to construct certain symbols whose index multiplicity gives $\mathcal P_X^F$.

\smallskip

\subsection{Some explicit computations  for $K$-theory}\label{Ktheory}

Let $Y$ be a sequence of vectors in $\Lambda $
and $M_Y=\oplus_{a\in Y}L_a$  the corresponding complex $G$-representation space.
We write $z=\oplus z_a$ an element of $M_Y$, with $z_a\in L_a$.

We choose a
 $G$-invariant Hermitian structure $<,>$ on $M_Y$.

We first recall here the description of  the generator $\Bott(M_Y)$ of  $K_G^0(M_Y)$, a free module of rank $1$  over $R(G)$.
 Let $E:=\bigwedge M_Y$ with the Hermitian structure induced by that of $M_Y$, graded as $E^+\oplus E^-$ by even and odd degree.
Then, for $z\in M_Y$, consider the exterior multiplication  $m(z):E\to E$, $ m(z)(\omega):=z\wedge \omega,$ and the Clifford  action
\begin{equation} \label{cli}
c(z)=m(z)-m(z)^*,
\end{equation}
of $M_Y$ on $E$. One has
$c(z)^2=-\|z\|^2$
so that $c(z)$ is an isomorphism, if $z\neq 0$,  exchanging the summands $E^+$ and $ E^-$ .
Consider the complex $G$-equivariant vector bundles $\CE^{\pm}=M_Y\times E^{\pm}.$
The $G$-equivariant morphism  from $\CE^+$ to $\CE^-$ defined by
$\Sigma_z(\omega)=c(z)\omega$ is supported at $0$, thus defines an element $\Bott(M_Y)$ of $K_G^0(M_Y)$,   generator of $K_G^0(M_Y)$ over $R(G)=\mathbb Z[\Lambda]$.

Notice that, if $Y=Y_1\cup Y_2$,
$\bigwedge  M_Y=\bigwedge M_{Y_1}\otimes \bigwedge M_{Y_2}$  and  $\Bott(M_Y)$ is the external tensor product of the symbols $\Bott(M_{Y_1})$ and $\Bott(M_{Y_2})$.

\begin{definition}\label{def:sigmaY}
 Given a $G$-invariant Hermitian structure $<,>$ on $M_Y$,  we define the $G$-invariant real one form
$$\sigma_Y=-\frac{1}{2}\Im <z,dz>.$$

\end{definition}

Here $\Im:\C\to \R $ is the imaginary part.

For example if  $M=L_a\sim \R^2$ and $z=v_1+iv_2$, then  $\sigma=\frac{1}{2} (v_1dv_2-v_2dv_1)$.

We consider the moment map $\mu_Y$ for $\sigma_Y$.
Thus $$\mu_Y(z)=\frac{1}{2}\sum_{a\in Y} |z_a|^2 a.$$

\begin{definition}\label{def:Z}
We define $Z_Y$ to be the zero fiber of the moment map $\mu_Y$:

$$Z_Y:=\{z\in M_Y\,|\, \sum_{a\in Y} |z_a|^2 a=0\}.$$

\end{definition}

This set $Z_Y$ was also denoted $M_Y^0$ in \cite{dpv4}.
However, as when $Y=X\cup -X$, we will  use different moment maps on $M_Y$, we keep the notation $Z_Y$ for the set of zeroes of the moment map $\mu_Y$ defined above.

The  following construction of some elements of $K_G^0(Z_Y)$ is due to Boutet de Monvel.

Of course the Bott element $\Bott(M_Y)$ restricts to $Z_Y$ as an element of $K_G^0(Z_Y)$. Let us construct some genuine elements not coming from the space  $K_G^0(M_Y)$.

We start with a simple case.  Assume  first that there is an element   $\phi\in \g$   which is  strictly positive on all the characters $a\in Y$. It follows that    $Z_Y=\{0\}$  and the equivariant $K$-theory   of $Z_Y$ is the  $\mathbb Z[\Lambda]$-module  generated by the class of the trivial vector bundle $\C$ over the point $Z_Y$ (which is in fact $\Bott(\{0\})$.

We come now to the case of an arbitrary sequence  $Y$  of weights.
 Let $F$ be a regular face for $Y$, we take a linear form $\phi\in F$, which  is non--zero on each element of $Y$.
We write $Y=A\cup B$, $A$ being the subsequence of elements on which $\phi$ takes positive values. $B$ being the subsequence of elements on which $\phi$ takes negative  values (notice that $A$ and $B$ depend only on $F$ and not on the choice of $\phi$). Accordingly we write $M_Y=M_A\oplus M_B$. Thus every $z\in M_Y$ can be uniquely decomposed as $z=z_A\oplus z_B$ with $z_A\in M_A$, $z_B\in M_B$.

Let $E_A=\bigwedge M_A$ graded as $E_A^+\oplus E_A^-$ taking the odd and even degree parts.

\begin{definition} The morphism
    $\Sigma^F$ between  the trivial bundles $M_Y\times E_A^+$ and $M_Y\times E_A^-$ is defined  by $$\Sigma_z^F=c(z_A)\ \ \forall z\in M_Y.$$
\end{definition}
It is clear that the support  of $\Sigma^F$ is the subspace $M_B$. We deduce
\begin{lemma}\label{zeridimu}   The intersection of the support of $\Sigma^F$  with $Z_Y$
      reduces to the zero vector.
      \end{lemma}
\begin{proof}     Indeed, in $Z_Y$,  $\sum_{a\in A} |z_a|^2a= -\sum_{b\in B} |z_b|^2 b$. If we are in the support of $\Sigma^F$  each $z_a=0$ so the we deduce that $ -\sum_{b\in B} |z_b|^2 b=0$. Since    $\phi$ takes a negative value  on each $b\in B$, this implies that $z_b=0$ for all $b$.
\end{proof}
We deduce that the restriction  of $\Sigma^F$ to $Z_Y$
defines an element of $K_G^0(Z_Y)$ still denoted by $\Sigma^F$.
\smallskip
%In $K$-theory, we have the identity
%$$(\Lambda M_B) \Sigma^F=\Bott(M_Y).$$

%

%

%

%

%

Let us now consider the case   $Y=X\cup -X$. In this case   $M_Y=M_X\oplus M_{X}^*=T^*M_X$ so that $\Bott(T^*M_X)$ gives a class
 in $K^0(T^*M_X)$, whose index
 is the trivial representation of $G$ (see   \cite{Ats1}).

We may restrict $\Bott(T^*M_X)$ to
$T^*_GM_X$ getting a class whose index multiplicity is  the delta  function $\delta_0$ on $\Lambda$.

In order to get further elements of $K_G^0(T^*_GM_X)$ we follow the construction of   Boutet de Monvel as follows.
\smallskip

Consider the $\mathbb R$--linear $G$-- isomorphism $h$ of $M_X^*$ with $M_{X}$ given by
$$ \xi(z):=\Re(<z ,h(\xi)>).$$
Here $ \Re(<z_1,z_2>)$ is the real part of the Hermitian product on $M_X$, a positive definite inner product.
       The isomorphism $h$ induces a
       $\mathbb R$--linear $G$-- isomorphism, still denoted by $h$, of $T^*M_X$ with $M_{X}\oplus M_X$.

Let $F$ be a regular face of the arrangement $X$ (and hence also of $Y=X\cup -X$).
 Let $\phi\in F$ so that $X=A\cup B$, with $\phi$ positive in $A$ and negative on $B$. We denote by $J$ the standard complex structure on $M_X$ and by $J_F$ the complex structure on $M_X$ defined as $J_F$ is $J$ on $M_A$ and $-J$ on $M_B$.
Then the list of weights of $G$  for this new complex structure on $M_X$  is $A\cup -B$.

 We consider the associated one form  $\sigma^F$  on $M$ which has moment map $$\nu_F(z)=\frac{1}{2}(\sum_{a\in A} |z_a|^2 a -\sum_{b\in B} |z_b|^2 b).$$ Clearly the zero fiber is reduced to $\{0\}$.

\begin{lemma}\label{changmom}
Consider the isomorphism of $T^*M_X$ with $M_X\oplus M_X$ given by
$$(z,\xi)\to [h(\xi)+J_Fz,h(\xi)-J_Fz].$$

In this isomorphism
the moment map $\mu$ on $T^*M_X$ associated to the Liouville form
    becomes the moment map $[\frac{1}{2}\nu_F,-\frac{1}{2}\nu_F]$  for $[\frac{1}{2}\sigma_{F},-\frac{1}{2}\sigma_{F}]$.

    In particular under the above isomorphism, the space $T^*_GM_X$ is identified with the zeroes of the moment map   $[\frac{1}{2}\nu_F,-\frac{1}{2}\nu_F]$.

\end{lemma}

We define the map \begin{equation}\label{mappaF}p_F: T^*M_X\to M_X\end{equation} by $p_F(z,\xi)=h(\xi)+J_F z$.
\begin{definition}\label{At} We set
$$\Sigma^F(z,\xi)=c(h(\xi)+J_F z).$$

In other words, $\Sigma^F=p_F^* \Bott(M_X)$ is the pull-back of the morphism $\Bott(M_X)$ by $p_F$.

\end{definition}

By Lemma \ref{changmom} and  Lemma \ref{zeridimu}, the intersection of the support of $\Sigma^F$ with $T_G^*M$ is reduced to the zero vector. Thus $\Sigma^F$
determines an element of $K_G^0(T^*_G M_X)$
 which  depends only of the connected component $F$  of $\phi$ in the set of regular elements.

%\marginpar{added more on definition of $\rho$}
Denote by $\rho$ the representation of $G$ in $M_X$, and  also by $\rho$ the infinitesimal action of $\g$ in $M_X$. If $\phi\in \g$, and $z=\sum_a z_a$ is in $M_X$, then $\rho(\phi)z=\sum_a i\langle a,\phi\rangle z_a.$
\begin{lemma} The symbol $\Sigma^F$ is equal in $K$-theory to the Atiyah symbol $$At^F(z,\xi)=c(h(\xi)+\rho(\phi) z)$$ (see \cite{dpv2}).\end{lemma}\begin{proof}
 Indeed, for $t\in [0,1]$,  it is easy to see that the
  $$At^F_t(z,\xi)=c(h(\xi)+(t\rho(\phi)+(1-t)J_F) z)$$
is  transversally elliptic. Thus $At^F$ and $\Sigma^F$ being homotopic coincide in $K$-theory.\end{proof}

%The  support of  $At^F^t(z,\xi)$ is obtained by $h(\xi)=-(t\rho(\phi)+(1-t)J_F)z$.
%If $(z,\xi)$ is in $T_G^*M$, we thus have
%$<h(\xi),\rho(\phi)(z)>=0$.
%But this is   $< -(t\rho(\phi)+(1-t)J_F)z,\rho(\phi)z>$,
%the strictly negative quantity $\sum_{a\in A}  (-t \ll a,\phi\rr ^2+(1-t)\ll a,\phi\rr) [z_a|^2+\sum_{b\in B} (-t \ll a,\phi\rr ^2-(1-t)\ll b,\phi\rr)|z_b|^2.$
%
%
%
%

Let us consider $$\Theta_X^F(e^x)=\sum_{\lambda}\mathcal P_X^F(\lambda) e^{i\ll \lambda,x\rr }$$ as a generalized function on $G$.

Recall Atiyah theorem, \cite{At}, \S 6 (see also Appendice 2 \cite{BV1}).
\begin{theorem}\label{indexAt}
$$\index(At^F)(g)=(-1)^{|X|}g^{a_X}\Theta_X^F(g)=\Theta_{-X}^F(g).$$
\end{theorem}

We translate immediately this theorem  as follows.

 \begin{theorem}\label{ATin}
 Let $F$ be a regular face.
  Let $$\Sigma^F= p_F^*\Bott(M_X).$$
 Then $\ind_m(\Sigma^F)=\mathcal P_{-X}^{F}$

 \end{theorem}

This identity agrees with some simple remark.
\begin{remark}
In the $K$-theory of $K_G^0(T^*_GM)$, we have the identity
$$(\Lambda M_{-X}) \Sigma^F=\Bott (M_X\oplus M_{-X}).$$
Applying the index, we obtain the identity
$$\prod_{a\in X}(1-e^{-i \ll a,x\rr })\index(At^F)(e^x)=1.$$
 \end{remark}

%\newpage
\subsection{The infinitesimal index}\label{indinfe}

 Consider $N$ an oriented  $G$-manifold, equipped  with a $G$-invariant $1$-form $\sigma$. Recall that $Z$ is the set of zeroes of the moment map $\mu: N\to \g^*$.

We denote by $\iota_x$ the contraction of a differential form on $N$ by the vector field $v_x$ associated to $x$ on $N$.
We have defined in \cite{dpv3} a Cartan model for the equivariant cohomology with compact supports $H_{G,c}^*(Z)$  of the subset $Z$ of $N$.  This is a $\Z$-graded space.  A representative of this group is an equivariant form $\alpha(x)$ with compact support: $\alpha:\g\to \mathcal A_c(N)$ such that $D(\alpha)$ is equal to $0$ in a neighborhood of $Z$.  The dependance of $\alpha$ in $x$ is polynomial.
Here   $\mathcal A_c(N)$ is the space of differential forms with compact supports on $N$, while  the equivariant differential $D$ is defined by $$D\alpha(x)=d\alpha(x) -\iota_x\alpha(x).$$

Clearly, an element $\alpha\in H_{G,c}^*(N)$ of the  equivariant cohomology with compact supports of $N$ defines a class in
$H_{G,c}^*(Z)$. Indeed $D\alpha=0$ on all $N$.

\begin{remark}\label{rem:almostcompact}

If     $\alpha$ is an equivariantly closed form  on $N$  such that the support of $\alpha$ intersected with
$Z$ is a compact set, we associate to $\alpha$ an element $[\alpha]_c$ in
the  equivariant cohomology with compact supports of $Z$ defined as follows.
Take a $G$-invariant function $\chi$ equal to $1$ in a neighborhood of $Z$ and supported sufficiently near $Z$, then $\chi\alpha$ is compactly supported on $N$ and $D\alpha=0$ in the neighborhood of $Z$ where $\chi=1$, thus
 defines a class $[\alpha]_c$ in $H_{G,c}^*(Z)$ independent of the choice of $\chi$.

\end{remark}

\bigskip

Let us now recall the definition of the infinitesimal index of
$[\alpha]\in H_{G,c}^*(Z).$
Let $\Omega=d\sigma$ and set
 $\Omega(x)=D \sigma(x)=\mu(x)+\Omega$. $\Omega(x)$ is
a closed (in fact exact) equivariant  form on $N$.
If $f$ is a smooth function on $\mathfrak
g^*$ with compact support, let
  $$\hat f(x):= \int_{\mathfrak g^*}e^{-i\langle \xi\,|\,x\rangle}f(\xi)d\xi$$
  be  the Fourier transform of $f$.
Choose  the measure  $dx$  on $\mathfrak g$  so that  the inverse Fourier
transform is $
f(\xi)=\int_{\mathfrak g}e^{i\langle \xi\,|\, x\rangle}\hat
f(x)dx$, thus $\hat f(x)dx$  is
independent of the choices.

 The double integral
$$\int_{N}\int_{\mathfrak g} e^{is\Omega(x)}\alpha(x) \hat f(x) dx$$
is independent of $s$ for $s$ sufficiently large.

We then defined:
\begin{equation}\label{definf}
 \langle \infdex_G^{\mu}([\alpha]),f\rangle :=\lim_{s\to \infty} \int_{N}
 \int_{\mathfrak g}
 e^{is \Omega(x)} \alpha(x) \hat f(x) dx.
 \end{equation}
This is a well defined map from $H_{G,c}^*(Z)$ to   distributions on $\mathfrak g^*$.
     It is a map of $S[\g^*]$ modules, where $\xi\in \g^*$  acts on forms by  multiplication by   $\ll \xi,x\rr$  and on distributions  by $i\partial_{\xi}$.

As the notation indicates, $\infdex_G^{\mu}$ depends only on $\mu$, and  not on the choice of the real invariant form $\sigma$ with moment map $\mu$.
We can obviously also change $\mu$ to a positive multiple of $\mu$ without changing the infinitesimal index. However, the change of $\mu$ to $-\mu$ usually changes radically the map $\infdex$.

 \begin{remark}\label{globalcompact}

Clearly, if $\alpha$ is  a closed equivariant form with compact support on $N$,    then    the infinitesimal index is just the Fourier transform   of the equivariant integral
 $ \int_{N} \alpha(x)$ of the equivariant form $\alpha$, as $e^{is\Omega(x)}=e^{isD\sigma(x)}$ is equivalent to $1$ for any $s$.

 \end{remark}

     We have extended,   using the same formula (\ref{definf}), the definition of $\infdex_G^{\mu}$ to a $\Z/2\Z$ graded cohomology space
      $\mathcal
H^{\infty,m}_{G,c}(Z)$ that we now define.

A representative of  a class $[\alpha]$ in
$\mathcal
H^{\infty,m}_{G,c}(Z)$ is a smooth map $\alpha:\g \to \mathcal A_c(N)$, such that the dependance of $\alpha(x)$  in $x$ is of at most polynomial growth, and  such that $D(\alpha)$ is equal to $0$ in a neighborhood of $Z$.
      The index $m$ indicates the moderate growth  on $\mathfrak g$ .

As we will recall in Subsection \ref{chern} , the equivariant  Chern character $\ch(\Sigma)$  of an element $\Sigma\in K_G^0(Z)$ belongs to the space $\mathcal
H^{\infty,m}_{G,c}(Z)$.

We note the following.
\begin{proposition} Let $b(x)=\int_{\g^*}e^{i\ll \xi,x\rr} m(\xi)d\xi$ be the Fourier transform of a    compactly supported distribution $m(\xi)$  on $\g^*$. Then
$b(x)$ is a   function on $\g$ of moderate growth  and
the space
$\mathcal
H^{\infty,m}_{G,c}(Z)$ is stable by multiplication by $b(x)$.

Furthermore

\begin{equation}\label{eq:convol}
\infdex^{\mu}_G(b \alpha)=m*\infdex_G^{\mu}(\alpha).
\end{equation}
\end{proposition}
 The character group $\Lambda$ acts on forms with moderate growth by multiplication by $e^{i\langle \lambda,x\rangle}$, for $\lambda\in \Lambda$, inducing an action on cohomology. It also acts on distributions on $\mathfrak g^*$ by translations:
$$t_{\lambda} D(f)=D(t_{-\lambda}f)=D(f(\xi+\lambda)).$$

%\marginpar{I think the sign was not correct}
\begin{proposition}\label{equiinfdex} The map $infdex$ is equivariant with the respect to the previous actions of $\Lambda$.
\end{proposition}
\begin{proof} The proposition follows from the definition of $infdex$ once we notice that
$$\widehat{t_{-\lambda}f}=e^{i\langle \lambda,x\rangle}\hat f$$
for any function $f$  on $\mathfrak g^*$ lying in the Schwartz space.
\end{proof}

\subsection{Some explicit computations in
 cohomology}\label{coho}

Recall that if $N$ is a vector space provided with an action of $G$,   $H_{G,c}^*(N)$ is a free $S[\g^*]$ module with a generator $\Thom(N)$ with equivariant integral $\int_N\Thom(N)(x)$ identically equal to $1$ (see for example \cite{par-ver2}).

Notice that in particular, by Remark \ref{globalcompact},  if $N$ is a   vector space with a  (any) real one form $\sigma$, then $\Thom(N)$ defines  an element of $H_{G,c}^*(Z)$ with infinitesimal index equal to the $\delta$-function of $\g^*$.
The form $\Thom(N)$ depends of the choice of an orientation of $N$.

Let $Y$ be a sequence of vectors in $\Lambda$ and
  $M:=M_Y$ be the corresponding complex $G$-representation space.
We give to $M$ the orientation given by its complex structure.

We want to describe $\Thom(M_Y)$.
For this it is sufficient to give the formula of the Thom form for $M=L_a$, a complex line.
The formula for $\Thom(M_Y)$ is then obtained by taking the exterior product of the corresponding equivariant differential forms.

If $\alpha(x)$ is an equivariant form on  a $G$-manifold $M$, fixing $m\in M$, we may write
$\alpha[m](x)$ for the element $\alpha(x)_m\in \bigwedge T_m^*M$
defined by the differential form $\alpha(x)$ at the point $m$.

 Let  $L_a=\C$ with coordinate $z$.
 The infinitesimal action of $x\in \g$  is given by $i\ll a,x\rr $.
Choose a function $\chi$  on $\R$ with compact support and identically $1$   near $0$.
Then
\begin{equation}\label{exa:Vdim2}\Thom(L_a)[z](x)= -\frac{1}{2\pi}\Big(\chi(|z|^2) \ll a,x\rr +
  \chi'(|z|^2) i(dz\wedge d\overline z)\Big)\end{equation}
  is the required closed equivariant differential form on $L_a$ with equivariant integral identically equal to $1$.

As in section \ref{Ktheory} we fix an $G$-invariant Hermitian structure on $M_Y$ and take the $G$-invariant real one form
 $\sigma_Y=- 1/2\Im <z,dz>$ with
 moment map  $\mu_Y(z)=\frac{1}{2}\sum_{a\in Y} |z_a|^2 a $  and zero fiber
 $Z_Y$.

We define some elements of $H_{G,c}^*(Z_Y)$ by an analogous procedure to the $K$-theory case. Of course, the restriction of $\Thom(M_Y)$ to $Z_Y$ defines an element in $H_{G,c}^{*}(Z_Y)$.

 Let us construct some further elements of $H_{G,c}^*(Z_Y)$, not coming by restriction to $Z_Y$ of a compactly supported class on $M_Y$.

As before we start with the basic case in which there is  $\phi\in \g$    strictly positive on all the characters $a\in Y$, then  $Z=\{0\}$  and its equivariant cohomology with compact supports is the algebra $S[\g^*]$ generated by $1$.

Do not assume any more that the weights $Y$ generate a pointed cone.
 Let $F$ be a regular face for $Y$.
  Let $A$ be the subsequence of $Y$ where $\phi$ takes positive values
and  $B$ be the subsequence where $\phi$ takes negative values.
 Write $M_Y=M_A\oplus M_B$.
  Then the pull back of   $\Thom(M_A)$  by the projection  $M_Y\to M_A$ is supported near $M_B$. Thus the support of the pull back of  $\Thom(M_A)$ intersected with $Z_Y$ is
  compact, and therefore,  as explained in Remark \ref{rem:almostcompact}, $\Thom(M_A)$ defines a class in $H_{G,c}^*(Z_Y)$. We can write an explicit    representative of this class by choosing a $G$ invariant function $\chi$    on $M_Y$ identically  equal to $1$  in a neighborhood of $Z_Y$ and supported near $Z_Y$. Then our representative will be given by
 $$t^F[z](x):=\chi(z) \Thom(M_A)[z_A](x).$$
\smallskip

Consider the inclusion $i_B:M_B\to M_Y$.
In \cite{dpv4} we  have defined a map $(i_B)_!:H_{G,c}^*(Z_B)\to H_{G,c}^*(Z_Y)$
 preserving the infinitesimal index.
 In this setting we see that the class $t^F\in H^*_{G,c}(Z_Y)$ is  by definition a representative of $(i_B)_!(1)$.

We now compute the infinitesimal index
$\infdex^{\nu}_G([\alpha])$
of the elements $t^F$.

The equivariant Thom form has equivariant integral equal to $1$, so that  by Fourier transform
$$\infdex^{\nu}_G(\Thom(M_Y))= \delta_0$$
                                      where $\delta_0$ is the $\delta$-distribution on $\g^*$.

Consider now the element $t^F$ associated to $F$.
The subsequence $B$  spans a pointed cone.
  We can then define the partial multispline distribution $T_B$ on $\g^*$.

\begin{theorem}\label{musp}  $\infdex^{\nu}_G(t^F)=(2i\pi)^{|B|}T_B$, where
 $T_B$ is the multivariate spline.
\end{theorem}

\begin{proof}

 Since the class $t^F\in H^*_{G,c}(Z_Y)$ is  a representative of $(i_B)_!(1)$and $i_!$ preseserves $infdex$, we are reduced to prove our theorem when $M_A=0$. The computation is reduced to the one dimensional case by the product axiom, \cite{dpv4}.  Let us make the computation  in this case, that is when  $M_Y=L_b$.

 The action form $\sigma$ in coordinates $z=v_1+iv_2$ is
$\frac{1}{2} (v_1 dv_2-v_2 dv_1)$, so $D\sigma(x) =dv_1\wedge
dv_2+\frac{1}{2}\ll b,x\rr \|v\|^2$.

Let $\chi(t)$  be a function on $\R$ with compact
support and  identically equal to $1$ in a neighborhood of $t=0$.
Then by definition,
 we get
$$\langle \infdex^{\nu}_G(1),f\rangle =
\lim_{s\to \infty} \int_{\R^2} \int_{-\infty}^\infty\chi(\frac{\|v\|^2}{2}) e^{i s \ll b,x\rr
\frac{\|v\|^2}{2}}e^{i s dv_1 dv_2}\hat f(x) dx $$
$$=\lim_{s\to \infty} is \int_{\R^2} \int_{-\infty}^\infty\chi(\frac{\|v\|^2}{2}) e^{i s \ll b,x\rr
\frac{\|v\|^2}{2}} dv_1 dv_2\hat f(x) dx. $$

We  pass in polar coordinates on $\R^2$, and  take $t=\frac{\|v\|^2}{2}$ as new variable. We obtain
$$\lim_{s\to \infty} 2i\pi s\int_{t=0}^{\infty} \int_{x=-\infty}^\infty\chi(t) e^{i \ll st b,x\rr }\hat f(x) dx dt.$$

Change $t$ in $t/s$, and use the Fourier inversion formula, we obtain

$$\langle \infdex^{\nu}_G(1),f\rangle =2i\pi\lim_{s\to \infty}\int_0^{\infty}\chi(t/s) f(tb) dt.$$

Passing to the limit, as $\chi$ is identically $1$ in a neighborhood of $0$ and $f$ is compactly supported, we obtain our formula
$$\langle \infdex^{\nu}_G(1),f\rangle =2i\pi \int_0^{\infty}  f(tb) dt.$$

\end{proof}

Exactly as in the $K$-theory case, when $Y=X\cup-X$, that is $(T^*M_X)^0=T^*_GM_X$, we can use this construction  to get classes in
$t^F\in H^*_{G,c}(T^*_GM_X)$ and compute their $\infdex$.
Let $F$ be a regular face for $X$. Consider the map $p_F: T^*M_X\to M_X$ defined by  $p_F(z,\xi)=h(\xi)+J_F z$ (cf. Section \ref{Ktheory}). We have:
%\marginpar{I think there is a sign}
\begin{theorem}

Let $\chi$ be a $G$-invariant function on $T^*M_X$ identically equal to $1$ in a neighborhood of $T^*_G M_X$ and supported near $T^*_GM_X$.
Then
$$t^F=\chi\, p_F^*\Thom(M_X)$$
defines a class in $H^*_{G,c}(T^*_GM_X)$
 such that $\infdex^{-\mu}_G(t^F)=(-1)^{|X|} (2i\pi)^{|X|} T_{X}^F$.

\end{theorem}

\begin{proof}

This is obtained from the preceding calculations.
Indeed the infinitesimal index depends only on the moment map. So, using Lemma \ref {changmom},   we are reduced to the calculation performed in Theorem \ref{musp}. The sign comes from taking in account  the orientations on $T^*M_X$. Indeed in the isomorphism of $T^*M_X$ with $M_X\oplus M_X$, the orientation of $T^*M_X$ is $(-1)^{|X|}$ the orientation given by the complex structure on $M_X\oplus M_X$.

\end{proof}

\section{The equivariant Chern character and the index theorem}
In this section, we compare the equivariant $K$-theory and the equivariant cohomology via the Chern character.

\subsection{Motivations}

Our goal is to compute the multiplicity index of a symbol $\Sigma\in K_G^0(T^*_GM)$  in terms of the infinitesimal index of the equivariant Chern character of  $\Sigma$ for a general $G$-manifold $M$.
We are going  to provide a direct formula at least in the case where $M=M_X$. This construction is motivated  by  taking the Fourier transform of the formula of Berline-Vergne for the equivariant index of a transversally elliptic operator (\cite{BV1}, \cite{BV2}, \cite{par-ver4} where one can also find the various notations and definitions).  We  first recall  this formula  in the simple case of elliptic symbols.

In this case the equivariant  Chern character $\ch(\Sigma)$  of an element $\Sigma\in K_G^0(T^*M)$ is an element in $\mathcal H_{G,c}^{\infty}(T^*M)$ and the index of $\Sigma$ is a regular function on $G$. For $x\in \g$ small enough,
\begin{equation}\label{delocalized}
\index(\Sigma)(e^x)=(2i\pi)^{-\dim M}\int_{\mathbf{T}^*M} \ch_{c}(\Sigma)(x) \hat A(T^*M)(x),
\end{equation}
where  $\hat A(T^*M)(x)$ is the equivariant $\hat A$ genus of $T^*M$,  $ \ch_{c}(\Sigma)(x)$  is the Chern character with compact support and   $\hat A(T^*M)(x)$  is  defined for $x$ small enough.
For any element $g\in G$, similar "descent formulae" are given for
$\index(\Sigma)(ge^x)$
where the integral is over $T^*M^g$, $M^g$ being the fixed point set of the action of $g\in G$ on $M$.

Let $\omega$ be the canonical (Liouville) $1$-form on $T^*M$:
$\omega_{m,\xi}(V)=\ll \xi,p_*V\rr$.
           Let $\Sigma$ be a transversally elliptic symbol.
In   \cite{BV1}, \cite{BV2},  it is  shown that, although $\ch(\Sigma)$ is not compactly supported, the formula

\begin{equation}\label{delocalizedote}
\index(\Sigma)(e^x)=(2i\pi)^{-\dim M}\int_{\mathbf{T}^*M} e^{-iD\omega(x)}\ch(\Sigma)(x) \hat A(T^*M)(x)
\end{equation}
still holds as a generalized function of $x$, for a sufficiently large class of transversally elliptic symbols.
   The factor  $e^{-iD\omega(x)}$ is congruent to $1$ in cohomology, but is crucial in defining a convergent oscillatory integral when $\ch(\Sigma)(x)$ is not compactly supported.

Let us  write  more explicitly Formula
(\ref{delocalizedote}) in the case where $M$ is a manifold such that $T^*M$ is stably equivalent to  $M\times R$, a trivial vector bundle: here $R$ is a real representation space of $G$.
  Consider the sequence $X_R$ of weights of $G$ in $R\otimes \C$.
  Thus  if $a\in X_R$, then $-a\in X_R$.
Consider the function $$j_{R}(x)=\prod_{a\in X_R}\frac{1-e^{-i\ll a,x\rr }}{i\ll a,x\rr }$$ on
$\g$.
Then the equivariant class   $\hat A(T^*M)(x)$  is just the function $j_R(x)^{-1}$ and is defined only for $x$ small enough.
In this "trivial tangent bundle case",  Formula (\ref{delocalizedote}) implies that
$$j_R(x)\index(\Sigma)(e^x)=(2i\pi)^{-\dim M}\int_{\mathbf{T}^*M} e^{-i D\omega(x)}\ch(\Sigma)(x).$$

Recalling Formula (\ref{eq:FB})  and the definition of $\ind_m(\Sigma)$, we obtain
$$\hat B_{X_R}(x)\sum_{\lambda} \ind_m(\Sigma)(\lambda) e^{i\ll \lambda,x\rr }=
(2i\pi)^{-\dim M}\int_{\mathbf{T}^*M}  e^{-iD\omega(x)}\ch(\Sigma)(x).$$

Let $\mu$ be the moment map on $T^*M$ associated to the one form $\omega$.
Thus, by Fourier transform, this equality suggests that   the following  formula should hold:

\begin{equation}\label{eq:boxindex}
B_{X_R}*_d \ind_m(\Sigma)=(2i\pi)^{-\dim M}\infdex_G^{-\mu}( \ch(\Sigma)).
\end{equation}

All the terms of this formula make sense, since we will show in Subsection \ref{chern} that
$\ch(\Sigma)$ belongs to the cohomology  $\mathcal H_{G,c}^{\infty,m}(T^*_GM)$ of classes with compact support on
$T^*_GM$ where the infinitesimal index is defined.

Consider now $M=M_X$, our representation space for $G$. We always will consider $M_X$ as a real $G$-manifold, except if specified differently.
 Let $T^*M_X=M_X\times M_X^*$ is a trivial vector bundle. Here $M_X^*$ is thus considered as the real vector space dual to $M_X$.
  The sequence of weights of $G$ in $M^*_X\otimes \C$
is the sequence $ X\cup (-X).$ Note that the real dimension of $M_X$ is $2|X|$.
Our aim in the next subsections is to  see   that Formula (\ref{eq:boxindex})  indeed  holds for any   $\Sigma\in K_G^0(T^*_G M_X)$.

Finally using "descent formulae" for $g$ running over the finite set of toric vertices of the system $X$, we will use Theorem \ref{theo:main} to give a formula for the multiplicity function $\ind_m(\Sigma)$ on $\Lambda$.

Let us point out  that we could use the formulae of \cite{BV1},\cite {BV2} or \cite{par-ver4}. However we found instructive to prove Formula (\ref{eq:boxindex}) directly in the case of $M_X$ is a vector space  using the explicit description of the generators of $K_G^0(T^*_GM_X)$ given in\cite{dpv2}.
 Using the functoriality principle, we hope that  it will be  possible to describe directly  $\ind_m(\Sigma)$, a function on $\Lambda $, in function of $\infdex^{\mu}_G(\ch(\Sigma))$, a function on $\g^*$, for any transversally elliptic operator on a general $G$-manifold $M$.

\subsection{The equivariant Chern character}\label{chern}

We  recall the construction of $\ch(\Sigma)$ for $\Sigma$   a  morphism of vector bundles on   a general $G$-manifold $N$.
We refer to \cite{par-ver3} for the comparison between the  different constructions of the Chern character.

Let $\mathcal E$  be  a $G$-equivariant complex vector bundle on $N$.  We choose a $G$-invariant Hermitian structure and  a
$G$-invariant Hermitian connection $\nabla$ on $\mathcal E$. For any $x\in \mathfrak g$, let
 $\CL_x$ be the action of $x$ on the space $\Gamma(N,\bigwedge T^*N\otimes \CE)$ of  $\CE$ valued forms on $N$.
The operator $j(x):=\CL_x -\nabla_x$ is a bundle map called the {\em moment  }  of the connection $\nabla$.
At each point $n\in N$, $j(x)$   is an anti-hermitian endomorphisms of ${\mathcal E}_n$.
Let $F$ be the curvature of the connection $\nabla$, thus $F$ is a two-form on $N$  with values in  the bundle of  anti-hermitian linear operators on $\CE$.

The equivariant curvature of $\mathcal E$  at the point $n$ is by definition $j(x)+F$.
Then the equivariant Chern   character $\ch({\mathcal E}, \nabla)$ is the equivariant differential form
$$\ch({\mathcal E}, \nabla)(x)=\Tr( e^{j(x)+F}).$$ This is a closed equivariant differential form on $N$ with $C^{\infty}$ coefficients (see \cite{ber-get-ver}, chapter 7).

\begin{lemma}
Over any compact subset   of $N$,
the  Fourier transform of the equivariant Chern character
$x\to \ch({\mathcal E}, \nabla)(x)$  is  a compactly supported distribution on $\g^*$.
\end{lemma}

\begin{proof}
Let us fix $n\in N$. Set $E=\mathcal E_n$ the fiber of the vector bundle $\mathcal E$ at $n$ and $A=\bigwedge^{2*} T_n^*N $ the even part of the exterior of the cotangent space at $n$. Then  $(j(x)+F)(n)\in A\otimes \mathfrak u$ where $\mathfrak u$ is the Lie algebra of antihermitian linear operators on $E$.
The map $x\to j(x)$ defines a map $\g\to \mathfrak u$, with dual map $j^*: \mathfrak u^*\to \g^*.$

If  $P(E)$ denotes the projective space of $E$, we define
 $\mu^P: P(E)\to \mathfrak u^*$ by
 $$\mu^P(p)(iX)=\frac{\ll X v,v\rr}{\ll v,v\rr }.$$
Here $p$ is the point of $P(E)$ associated to $v\in  E-\{0\}$ and $iX\in \mathfrak u$.

By Corollary \ref{finalg} to Nelson theorem (for completeness, in Theorem \ref{nelthe}  we give a proof of this fact    based on localization formula in equivariant cohomology),

$$  \Tr(e^{(j(x)+F)})=\int_{P(E)} e^{i \ll j^*\mu^P(p), x\rr} D(p,n)$$
where $D(p,n)=e^{\mu^P(p)(F)}\Tr(\beta(p,u))$ is  a differential form on $P(E)$ with values in $\bigwedge T_n^*N$
(if $F=\sum_k F_k u_k$ with $F_k\in \bigwedge^2 T_n^*N$ and $u_k\in \mathfrak u$,   $e^{\mu^P(p)(F)}= e^{\sum _k F_k \mu^P(p)(u_k)}$ is a smooth function on $P(E)$ with values in $\bigwedge T^*_nN$).

Integrating  over the fiber of the map $j^*\mu^P: P(E)\to \g^*$
we obtain

   $$\Tr(e^{ j(x)+F})=\int_{\g^*} e^{i\ll x,\xi\rr }  \gamma(\xi)$$
   where $\gamma(\xi)$  is a distribution supported on the compact set $j^*\mu^P (P(E))$.
   Thus we see thus that at each point $n$ of $N$, the function  $x\to \Tr(e^{j(x)+F})$ is the Fourier transform  of a compactly supported distribution on $\g^*$ (with values in $\bigwedge T^*_n N$). It is clear that our estimates are uniform on any compact neighborhood of the point $n\in N$. Thus we obtain  our lemma.

   \end{proof}

In particular the closed equivariant differential form $\ch(\Sigma)(x)$ has moderate growth with respect to $x\in \g$ over any compact subset of $N$.

\bigskip

Let ${\mathcal E}={\mathcal E}^+\oplus {\mathcal E}^-$ be  a Hermitian $G$-equivariant super-vector bundle over $N$.
Let ${\Sigma}:{\mathcal E}^+\to {\mathcal E}^-$ be a $G$-equivariant  morphism.
Outside the support of ${\Sigma}$, the complex vector bundles ${\mathcal E}^+$ and
${\mathcal E}^-$ are ``the same", so that it is natural to construct
representatives of $\ch({\mathcal E}):=\ch({\mathcal E}^+)-\ch({\mathcal E}^-)$ which are
zero ``outside'' the support of ${\Sigma}$ by the following
identifications of bundle with connections.

Let $U$ be a neighborhood of the support of  ${\Sigma}$.
We first may choose the Hermitian structures on $\CE^+,\CE^-$ so that $\Sigma$  is an isomorphism of hermitian vector bundles outside $U$.

A pair of connections $\nabla^+, \nabla^-$ is said ``adapted'' to
the morphism ${\Sigma}$ on $U$ when the following holds
\begin{equation}\label{eq:adapted}
    \nabla^- \circ {\Sigma} =  {\Sigma} \circ \nabla^+
\end{equation}
outside the  neighborhood $U$ of the support of ${\Sigma}$. A pair
of adapted
connections   is
easy to construct.

\begin{proposition}

    Let $\nabla^+, \nabla^-$ be a pair of  $G$-invariant Hermitian connections {\em adapted} to
    ${\Sigma}:{\mathcal E}^+\to {\mathcal E}^-$.
Then
    the differential form
    $\ch({\mathcal E}^+,\nabla^+)-\ch({\mathcal E}^-,\nabla^-)$ is a closed equivariant differential form on $M$
    supported near the support of ${\Sigma}$.
We note it as $\ch_s({\Sigma}).$

 \end{proposition}

The index $s$ means with support condition.

In particular if  the  support  of $\Sigma$ is compact, we can also choose the neighborhood $U$ so that its closure is compact. We deduce that the cohomology class of
$\ch_s({\Sigma})$ lies in  $\mathcal H^{\infty,m}_{G,c}(N)$ is a compactly supported class on $N$  with moderate growth in $x$ and is denoted simply by $\ch(\Sigma)$.

\bigskip

 Let $g\in G$ and let $N^g$ be the fixed point submanifold of $g$. Then $g$ acts by a fiberwise transformation on $\mathcal E\to N$ still denoted $g$.
 We still denote by $F$ the curvature of the bundle $\CE$ restricted to $N^g$.
The equivariant twisted Chern character $\ch^g({\mathcal E}, \nabla)$ is the equivariant differential form
$$\ch^g({\mathcal E}, \nabla)(x)=\Tr(g e^{(j(x)+F)}).$$ This is a closed equivariant differential form on $N^g$.
Similarly we have the following proposition.
\begin{proposition}\
    Let $(\nabla^+,\nabla^-)$ be a pair of  $G$-invariant Hermitian connections {\em adapted} to
    ${\Sigma}:{\mathcal E}^+\to {\mathcal E}^-$. Then the differential form
    $$\ch^g({\Sigma})(x)= \ch^g({\mathcal E^+}, \nabla^+)(x)-
    \ch^g({\mathcal E^-}, \nabla^-)(x)$$
    is a closed $G$-equivariant form on $N^g$ supported near the support of $\Sigma|_{N^g}$.

Over any compact subset of $N^g$, it has moderate growth with respect to $x\in \g$.
  \end{proposition}

If we change the choice of connections, we can see using the usual transgression formulae for Chern characters (see for example \cite{par-ver3}) that the class  $\ch^g(\Sigma)$ stays the same in the cohomology with moderate growth: that is, the boundary $\nu(x)$ expressing the change of  $\ch^g(\Sigma)$ with respect to the connection remains with moderate growth with respect to $x$, over any compact subset of $N^g$.

We return to the situation where $N=T^*M$ is the conormal bundle to a $G$-manifold $M$ and  $\Sigma$ is a transversally elliptic symbol. Let $\chi$ be a function identically equal to $1$  near the set $T^*_GM$ and supported in a neighborhood of $T^*_GM$ whose closure has compact intersection with the support of $\Sigma$. We have

\begin{proposition}

 The equivariant  form $\alpha^g(x)=\chi \ch^g({\Sigma})(x)$ on $T^*M^g$
is compactly supported and $D\alpha^g(x)$ is equal to $0$  in a neighborhood of $T^*_GM^g$. Over any compact subset of $T^*_G M^g$, it has moderate growth with respect to $x\in \g$.

\end{proposition}

It follows that  the Chern character  gives a morphism

$$\ch:K_G^0(T^*_GM)\to \mathcal H^{\infty,m}_{G,c}(T^*_GM).$$

Similarly for $g\in G$, the twisted Chern character is a morphism

$$\ch^g:K_G^0(T^*_GM)\to \mathcal H^{\infty,m}_{G,c}(T^*_GM^g).$$

As we have seen in Section \ref{indinfe}, the character group $\Lambda$ acts on forms with moderate growth by multiplication by $e^{i\langle \lambda,x\rangle}$, for $\lambda\in \Lambda$, inducing an action on cohomology. Of course $\Lambda$ also     acts on   $G$ equivariant vector bundles by tensoring with $L_a$  and inducing an action in $K$-theory. We have
%\marginpar{I wrote $L_a$ meaning the trivial bungle with fiber $L_a$ maybe this should be said here or before}
\begin{proposition}\label{equichern} For any $g\in G$ the map $ch^g$ is equivariant with the respect to the previous actions of $\Lambda$.
\end{proposition}
\begin{proof} Let $\mathcal E$  be  a $G$-equivariant complex vector bundle on $N$.  We choose a $G$-invariant Hermitian structure and  a
$G$-invariant Hermitian connection $\nabla$ with  moment $j(x)$ and curvature.
$F$.

Then for the vector bundle $L_\lambda\otimes \CE$, with endomorphism bundle   canonically isomorphic to the one of  $\CE$, we can take   the same connection $\nabla$.  By the definition of the moment, we see that the equivariant curvature   of $L_\lambda\otimes \mathcal E$  equals $i\ll \lambda,x\rr+j(x)+F$ giving our claim.
 \end{proof}

\subsection{ Explicit computations of the Chern character}

 We consider our $G$-manifold  $M=M_X$.
Choose an Hermitian structure on $M_X$.
Let $\Sigma_z=c(z)$ be the Clifford multiplication acting on the complex vector bundle $\bigwedge M_X$. The support of $\Sigma$ is $\{0\}$ and
$\Sigma$ determines the class $\Bott(M_X)\in K_G^0(M_X)$.

The following  result is well known.
$$\ch(\Bott(M_X))(x)= (2i\pi)^{|X|}\prod_{a\in X} \frac{e^{i\ll a,x\rr }-1}{i\ll a,x\rr }
        \Thom(M_X)(x)$$
in the cohomology group of smooth equivariant differential forms, without moderate growth conditions.

In fact, this equality  holds also in $\mathcal  H_{G,c}^{\infty,m}(M_X)$.

 \begin{proposition}\label{theo:chth}

 We have the equality
 $$\ch(\Bott(M_X))(x)= (2i\pi)^{|X|}\prod_{a\in X} \frac{e^{i\ll a,x\rr }-1}{i\ll a,x\rr }
        \Thom(M_X)(x)$$
in
$\mathcal H_{G,c}^{\infty,m}(M_X)$.

\end{proposition}

 \begin{proof} Since $M_X=\oplus_{a\in X}L_a,$ and both $\Bott(M_X)$ and $ \Thom(M_X)$ are the external product of the various $\Bott(L_a)$ and $ \Thom(L_a)$ for $a\in X$, it suffices to prove our claim when   $M_X=L_a$.

In this case
$\CE^+$ is the trivial bundle $L_a\times \C$,
$\CE^-=L_a\times L_a$, and the morphism is $\Sigma_z=z$.
We choose $\nabla^+=d$.

Let $\chi(t)$ be a function on $\R$ with compact support  contained in $|t|<1$ and identically equal to $1$ near $0$.
 Let $\beta=(\chi(|z|^2)-1)\frac{dz}{z}$, a well defined $G$ invariant $1$-form.
 We consider
$$\nabla^-=d+\beta$$
Outside $|z|^2<1$, the connections $\nabla^+=d, \nabla^-=d-\frac{dz}{z}$  verify $\nabla^-z=z\nabla^+$ so that the pair $(\nabla^+,\nabla^-)$ is adapted for the morphism $z$.

We compute the corresponding difference of Chern characters.

The moment $j(x)$ of the connection $\nabla^+$ is $0$ and the equivariant curvature $F^+(x)=0$.
So $\ch(\CE^+,\nabla^+)=1$.

 The moment of the connection $\nabla^-$ is $i\ll a,x\rr +i \ll a,x\rr (\chi(|z|^2)-1)=i\ll a,x\rr \chi(|z|^2)$. Thus
the equivariant curvature of $\nabla^-$ is
$$F^{-}(x)=i\ll a,x\rr \chi(|z|^2)-\chi'(|z|^2) dz\wedge d{\overline z}.$$

Remark that $F^{-}(x)=0$, if $|z|^2>1$, so that
$\ch(\CE^+,\nabla^+)-\ch(\CE^-,\nabla^-)$ is supported on $|z|^2<1$.
Thus $\ch(\Bott(L_a)):=\ch(\CE^+,\nabla^+)-\ch(\CE^-,\nabla^-)$
is a closed equivariant form with compact support.

We have explicitly
$$\ch(\Bott(L_a))[z](x)=(1-e^{i \ll a,x\rr \chi(|z|^2)})+e^{i\ll a,x\rr \chi(|z|^2)} \chi'(|z|^2) dz\wedge d{\overline z}.$$

Let us see that      $\ch(\Bott(L_a))[z](x)$ is  equal to
$$(2i\pi)\frac{e^{i\ll a,x\rr  }-1}{i\ll a,x\rr }\Thom(L_a)[z](x)=
\frac{e^{i\ll a,x\rr  }-1}{i\ll a,x\rr } (i \ll a,x\rr \chi(|z|^2)-\chi'(|z|^2) dz\wedge d{\overline z})$$
modulo a boundary with moderate growth.

We consider the one form $$\nu(x)=
\Big((\frac{e^{i\chi(|z|^2) \ll a,x\rr }- 1}{i\ll a,x\rr })-(\frac{e^{i\ll a,x\rr }- 1}{i\ll a,x\rr }) \chi(|z|^2)\Big)\frac{dz}{z}.$$

We see that $\nu(x)$ is  well defined and compactly supported on $L_a$. Indeed
$$(\frac{e^{i\chi(|z|^2) \ll a,x\rr }- 1}{i\ll a,x\rr })-(\frac{e^{i\ll a,x\rr }- 1}{i\ll a,x\rr }) \chi(|z|^2)$$ is equal to $0$ if $z$ is near $0$ where $\chi(|z|^2)$ is equal to $1$,  and is  also equal to $0$ when $[z|>1$ where $\chi(|z|^2)$ is equal to $0$.

Furthermore
the Fourier transform of $ {(e^{i\chi(|z|^2) \ll a,x\rr }- 1)}/{i\ll a,x\rr }$ is  supported, at the point $z\in L_a$,   on  the interval $[0,-\chi(|z|^2)a].$
Thus we see that $\nu$ has moderate growth.

Since it is easily verified that  $$D\nu(x)=2i\pi \frac{e^{i\ll a,x\rr  }-1}{i\ll a,x\rr }\Thom(L_a)(x)-\ch(\Bott(L_a))(x),$$
Proposition \ref{theo:chth} follows.

%To compute its class in $\mathcal H_{G,c}^{\infty}(L_a)$, it is sufficient to compute its integral. The integral of its   top degree term is
%$$\int_{\C} e^{ia(x)\chi(|z|^2)}\chi'(|z|^2) d\overline z dz.$$
%
%We pass in polar coordinates, so we need to compute
%
%$$2i\pi  \int_0^{\infty} e^{ia(x)\chi(t)}\chi'(t)dt=
%2i\pi  \int_0^{\infty} \frac{d}{dt} (\frac{e^{ia(x)\chi(t)}}{ia(x)})dt.$$
%
%As $\chi(0)=1$,
%we obtain that the integral is $$2i\pi\frac{e^{i a(x)}-1}{ia(x)}$$
%So
%$$\ch(\Bott(L_a))(x)= 2i\pi \frac{e^{i a(x)}-1}{ia(x)} \Thom(L_a)(x).$$
%
%
%
%GIVE THE COBOUNDARY...

\end{proof}

Let $g\in G$.
 The sub-manifold $M^g$ for the action of $g$ on $M$ is $M_{X^g}$ where
 $X^g:=[a\in X \,|\, g^a=1]$ a subsequence  of $X$.
Then the restriction of the symbol $\Bott(L_a)$ to $M^g$ is equal to
$$\Bott(M^g)\otimes \bigwedge(M_{X\setminus X^g}).$$

Thus we obtain

\begin{proposition}\label{chg}
We have
$$\ch^g(\Bott(M_X))(x)= (2i\pi)^{|X^g|}
\prod_{a\in X^g} \frac{e^{i\ll a,x\rr }-1}{i\ll a,x\rr }
\prod_{b\notin X^g}(1-g^be^{i\ll b,x\rr })\Thom(M_{X^g})(x).$$

\end{proposition}

We now compute the Chern character of the symbol $\Sigma^F$.
Recall the map $p_F(x,\xi)=h(\xi)+J_Fx$ from $T^*M$ to $M$. The morphism $\Sigma^F$
is the pull-back  of $\Bott(M_X)$ via this map.
It defines a class with compact support. Comparing with the element $t^F\in H_{G,c}^*(T^*_GM)$ which is  obtained as a pull-back of a Thom class, from Proposition \ref{theo:chth} we deduce

\begin{proposition}

We have the equality in $\mathcal H_{G,c}^{\infty,m}(T^*_GM)$

$$\ch(\Sigma^F)(x)= (2i\pi)^{|X|}\prod_{a\in X} \frac{e^{i\ll a,x\rr }-1}{i\ll a,x\rr } t^F(x).$$

  \end{proposition}

The computation of the infinitesimal index follows from this formula.
As $$
\prod_{a\in X} \frac{e^{i\ll a,x\rr }-1}{i\ll a,x\rr }=\int_{\g^*} e^{i\ll \xi,x\rr}B_X(\xi),$$
using  Formula (\ref{eq:convol}), we obtain:

%\marginpar{I think the constants were wrong due to orientations problems}

  \begin{theorem}\label{infdexchern}
$$\infdex^{-\mu}_G(\ch(\Sigma^F))= (2\pi)^{2|X|} B_X*_c T_X^F.$$

Similarly for any $g\in G$,  we have
$$\infdex^{-\mu}_G(\ch^g (\Sigma^F))=(2\pi)^{2|X^g|}
 \prod_{b\notin X^g} (1-g^bt_b)(B_{X^g}*_c T_{X^g}^F).$$

                    \end{theorem}

  In particular, when $g$ is a vertex on $X$ (Definition \ref{def:toricv}), the system $X^g$ spans $\g^*$ and  the infinitesimal index of   $\ch^g(\Sigma^F)$ is a piecewise polynomial function on $\g^*$ with respect to $(X,\Lambda)$.  In fact we even see that this functions are continuous on $\g^*$.

\subsection{The  index theorem }

We are now ready to compare   the morphism index and the morphism infdex on $T^*_GM_X$ and to prove Formula (\ref{eq:boxindex}).

We denote by $X_R$ the sequence $X\cup -X$ of characters. Remark that the zonotope associated to $X_R$ contains $0$ in its closure.

%\marginpar{I think the constants were skightly wrong}

\begin{proposition}\label{theo:boxindex}
Let $X\subset \Lambda$ be a  system of characters of $G$.
Let $$X_R=X\cup -X.$$
Let $\Sigma$ be a $G$-invariant transversally elliptic symbol on $M$.
Let $\ind_m(\Sigma)\in \mathcal C_\Z[\Lambda ]$  be its multiplicity index.
Let $\infdex^{-\mu}_G(\ch(\Sigma))$ be the infinitesimal index of its Chern character.
Then
$$B_{X_R}*_d \ind_m(\Sigma)=(2i\pi)^{-2|X|}\infdex_G^{-\mu}\ch(\Sigma).$$
\end{proposition}

\begin{proof} Using Proposition \ref{equiinfdex}  and Proposition \ref{equichern} ,
we are reduced to prove our  equality on generators of $K_G^0(T^*_G M_X)$.
We thus consider the symbol $\Sigma^F$  which, by Theorem \ref{ATin}, has  infinitesimal index equal to the polarized partition function
$\mathcal P_{-X}^F$.
 Recall that  $$B_{-X}*_d \mathcal P_{-X}^F=T_{-X}^F=(1)^{|X|} T_X^F.$$
As $B_{X_R}=B_X*B_{-X}$, the theorem follows from Theorem \ref{infdexchern}.
\end{proof}

Using the deconvolution theorem in the unimodular case, Proposition \ref{theo:boxindex} leads to the following theorem, which is strongly reminiscent of the Riemann-Roch theorem.
Remark that as $X_R$ contains $0$ in its interior, we may use any alcove containing $0$ in its closure in the limiting procedure.

We denote by $Todd(X_R)$ the Todd operator associated to $X_R$. It acts on the space of piecewise polynomial functions for the system $(X,\Lambda)$.

\begin{theorem}

Let $X\subset \Lambda$ be a unimodular system of characters of $G$.
Let $$Todd(X_R)=\prod_{a\in X\cup -X}\frac{\partial_a}{1-e^{-\partial_a}}$$
be the Todd operator.

Let $\Sigma$ be a $G$-invariant transversally elliptic symbol on $M$,
   $\ind_m(\Sigma)\in \mathcal C_\Z[\Lambda ]$   be  its multiplicity index and
 $\infdex^{-\mu}_G(\ch(\Sigma))$ be  the infinitesimal index of its Chern character.
Then

\begin{itemize}

\item $\infdex^{-\mu}_G(\ch\Sigma)$ is a piecewise polynomial measure on $\mathfrak g^*$.

\item

Let $\c$ be   an alcove   having 0 in its closure. We have

$$\ind_m(\Sigma)=(2i\pi)^{-2|X|}\lim_{\c} Todd(X_R)_{pw} \infdex^{-\mu}_G(\ch(\Sigma)).$$

\end{itemize}

\end{theorem}

\begin{remark} It is possible to show in this unimodular  case that the piecewise polynomial function   $(2i\pi)^{-2|X|} Todd(X_R)_{pw} \infdex^{-\mu}_G(\ch(\Sigma))$ extends to a continuous function on $\g^*$. Thus its restriction to $\Lambda$ gives the index multiplicity.

\end{remark}

We formulate now the general index theorem.
We denote by $\mathcal V(X)\subset G$ the set of toric vertices of the sequence of characters $X$ of $G$ (see Definition \ref{def:toricv}).

\begin{theorem}\label{final}
Let $X$ be a sequence of elements in $\Lambda$ and let $M:=M_X$. Let
$$X_R= X\cup -X.$$

Let $\Sigma$ be a $G$-invariant transversally elliptic symbol on $M$
and $\ind_m(\Sigma)\in \mathcal C_{\mathbb Z}[\Lambda ]$  be its multiplicity index.

For any $g\in \CV(X)$, let
 $\infdex^{-\mu}_G(\ch^g(\Sigma))$ be the distribution on $\g^*$ associated to  the cohomology class  $\ch^g(\Sigma)\in \mathcal H_{G,c}^{\infty,m}(T^*_G M^g)$ by the infinitesimal index.
 Then

\begin{itemize}

\item  $\infdex^{-\mu}_G(\ch^g(\Sigma))$  is a piecewise polynomial measure on $\mathfrak g^*$.

\item

Let $\c$ be   an alcove   having 0 in its closure. We have
% \marginpar{I split the formula in 2 lines before it was too long now not very nice}

\begin{equation}
\label{lamax}\ind_m(\Sigma)
=\sum_{g\in \mathcal V(X)}
 ( 2i\pi )^{-2|  X^g|}\hat g \lim_{\c} D(X_R\setminus X_R^g,g)^{-1}\hspace{3cm}\end{equation}
 $$\hspace{5cm} Todd(X_R^g)*_{pw}    \infdex^{-\mu}_G(\ch^{g^{-1}}(\Sigma)).$$

\end{itemize}
\end{theorem}

\begin{proof}

Again, using Proposition \ref{equiinfdex}  and Proposition \ref{equichern} ,
we are reduced to prove our  equality on generators of $K_G^0(T^*_G M_X)$.

Let $K=\ind_m(\Sigma^F)=(-1)^{|X|}t_{a_X}\mathcal P_X^F$.
From the general  inversion formula obtained in the first part of this paper (Theorem \ref{theo:main}),
we have  only to show that, for every $g\in \CV(X)$,

 $$B_{X_R^g}*_d \left(\hat g^{-1}\nabla_{X_R\setminus X_R^g}K\right)= (-1)^{|X\setminus X^g|}( 2\pi )^{-2|  X^g|}\infdex^{-\mu}_G(\ch^{g^{-1}}(\Sigma)).$$

%\marginpar{I think a sign was wrong}

Now observe that $$(-1)^{|X/X_g|}\nabla_{-X\setminus  (-X^g)}t_{a_X}\mathcal P_X^F=t_{a_{X^g}}\mathcal P_{X^g}^F$$

and
$$B_{-X^g}*t_{a_{X^g}}\mathcal P_{X^g}^F=T_{X^g}^F.$$
So substituing, we get
$$B_{X_R^g}*_d \left(\hat g^{-1}\nabla_{X_R\setminus X_R^g}K\right)=(-1)^{|X^g|} \prod_{b\notin X^g} (1-g^{-b}t_b)(B_{X^g}* T_{X^g}^F).$$
But  by Theorem
\ref{infdexchern}
$$\infdex^{-\mu}_G(\ch^{g^{-1}} (\Sigma^F))=(2\pi)^{2|X^g|}
 \prod_{b\notin X^g} (1-g^{-b}t_b)(B_{X^g}*_c T_{X^g}^F).$$
so our claim follows.

 \end{proof}

\subsubsection{The Box spline again}
It is quite amusing to verify this theorem on elliptic symbols.
The infinitesimal index of $\Bott(T^*M_X)$, after multiplying by $1/(2i\pi)^{2|X|}$  is the double box spline $B_{X\cup -X}.$
In the unimodular case, the "mother formula"
$$\lim_{\c}Todd(X\cup -X)*_{pw} B_{X\cup -X}=\delta_0$$
expresses just the fact that the index of the elliptic operator with symbol $\Bott(T^*M_X)$ is the trivial representation of $G$.

The infinitesimal index of the elliptic symbols are thus obtained by finite number of translations of the double box spline.

\appendix
\section{ Nelson formula}
We here sketch a short and explicit proof of Nelson formula, as suggested to us by Michel Duflo.

Let $E$ be a Hermitian vector space.
The projective space $P(E)$ is a Hamiltonian space for the action of the unitary group $U(E)$.
Let $\mathfrak u$ be the space of anti-hermitian matrices: $\mathfrak u$ is the Lie algebra of $U(E)$.

Let $\omega$ be the Kahler form on $P(E)$. We denote by $\omega(u)=\mu^P(u)+\omega$  the equivariant symplectic form.
Here $u\in \mathfrak u$, and $\mu^P(p)(u)=\frac{<uv,v>}{<v,v>}$, and $p\in P(E)$ is the image of $v$.
The form $\omega(u)$ is a closed equivariant form on $P(E)$.
We have
$$\frac{1}{(2i\pi)^{\dim P(E)}}
\int_{P(E)}e^{i\omega(u)}=1.$$

Consider the  $End(E)$-valued polynomial function of $(u,z)$
$$Q(u,z)=\frac{\det(u-z)}{u-z}.$$
 Here $u\in \mathfrak u$ and $z$ is a variable.
 We can substitute $\omega(u)$ to $z$, so that
 $\beta(p,u):=Q(u,\omega(u))$ is an $End(E)$ valued differential form on $P(E)$ depending polynomially of $u$.
 \begin{theorem}\label{nelthe}

 For any $u\in \mathfrak u$, we have
 \begin{equation}\label{Nelfor}\int_{P(E)} e^{i\omega(u)} \beta(p,u)=e^{iu}.\end{equation}

 \end{theorem}

 \begin{proof}
Since the formula \eqref{Nelfor} is clearly equivariant under conjugation and analytic in $u$,  it is sufficient to prove it  when $u$ is a generic diagonal matrix. In this case, the formula follows right away from the localization formula of Berline-Vergne applied to the action of the torus $\exp(tu)$.
 Let us for example do the calculation for $u$ a $3\times 3$ matrix (the general case is identical)

 $$u=\left(
       \begin{array}{ccc}
         i\theta_1 & 0 & 0 \\
         0 & i\theta_2 & 0 \\
         0 & 0 & i\theta_3 \\
       \end{array}
     \right)$$

     Then $$Q(u,z)=\left(
       \begin{array}{ccc}
        (i\theta_2-z) (i\theta_3-z) & 0 & 0 \\
         0 &(i\theta_1-z) (i\theta_3-z)  & 0 \\
         0 & 0 & (i\theta_1-z) (i\theta_2-z) \\
       \end{array}
     \right)$$

To compute the integral of
 $e^{i\omega(u)} \beta(p,u)$ over $P(E)$, we can apply the localization theorem. At the point $p_k=\C e_k$, $\omega(u)$ restricts to $i\theta_k$, thus the first  diagonal entry of $Q(u,\omega(u))(p_k)$ are zero except for $p_1=\C e_1$,  which is $(i\theta_2-i\theta_1) (i\theta_3-i\theta_1)$.
  Thus by the localization formula, the first diagonal entry of
  the matrix $\int_{P(E)}e^{i\omega(u)} \beta(p,u)$ is just $e^{-\theta_1}$. The calculation is similar for all diagonal entries.

  \end{proof}

Using the fact that formula \ref{Nelfor} is analytic in $u$ we immediately deduce
\begin{corollary}\label{finalg}
Let $A$ be a finite dimensional commutative algebra over $\mathbb R$ and $u\in A\otimes_{\mathbb R}\mathfrak u$. Then
 \begin{equation}\label{NelforA}\int_{P(E)} e^{i\omega(u)} Q(u,\omega(u))=e^{iu}.\end{equation}
\end{corollary}
\begin{proof}
We write $u=\sum_jx_jf_j$ where $f_j$ is a chosen basis for $\mathfrak u$. The two sides of the formula are power series in the variables $x_i$ which coincide for real values of the $x_i$ hence coincide formally and so we can substitute to the $x_i$ any commuting values.

 \end{proof}

\end{document}